\documentclass[a4paper,11pt,reqno]{amsart}

\usepackage{amsmath,amsfonts,amsthm,amssymb,color}
\usepackage[latin1]{inputenc}
\newcommand{\der}{\delta}

\newcommand{\cacha}{\Hat{\mathcal{C}}}

\newcommand{\delha}{\hat{\delta}}

\newcommand{\Laha}{\hat{\Lambda}}

\newcommand{\norm}[1]{\lVert #1\rVert}

\newcommand{\ka}{\kappa}

\newcommand{\iot}{\int_{0}^{t}}
\newcommand{\ist}{\int_{s}^{t}}

\newcommand{\Yti}{\widetilde{Y}}
\newcommand{\Wti}{\widetilde{W}}
\newcommand{\psiti}{\widetilde{\psi}}

\def\taun{\tau^{(n)}}
\def\ten{T^{(n)}}
\def\pn{^{(n)}}

\DeclareMathOperator{\id}{\text{Id}}

\makeatletter
\newcommand{\eqcolon}{\mathrel{\mathord{=}\raise.2\p@\hbox{:}}}
\newcommand{\coloneq}{\mathrel{\raise.2\p@\hbox{:}\mathord{=}}}
\makeatother

\newcommand{\RR}{\mathbb{R}}
\newcommand{\abs}[1]{\lvert #1\rvert}

\newcommand{\R}{\mathbb R}
\newcommand{\N}{\mathbb N}

\newcommand{\cb}{\mathcal B}

\newcommand{\cac}{\mathcal C}

\newcommand{\cd}{\mathcal D}

\newcommand{\cf}{\mathcal F}

\newcommand{\cj}{\mathcal J}
\newcommand{\cl}{\mathcal L}
\newcommand{\cn}{\mathcal N}
\newcommand{\cq}{\mathcal Q}

\newcommand{\cs}{\mathcal S}

\newcommand{\cp}{\mathcal P}

\newcommand{\al}{\alpha}
\newcommand{\ep}{\varepsilon}

\newcommand{\ga}{\gamma}
\newcommand{\la}{\lambda}

\newcommand{\oom}{\Omega}

\newcommand{\si}{\sigma}

\newcommand{\vp}{\varphi}

\newcommand{\be}{\beta}
\newcommand{\bbe}{\bar\beta}
\newcommand{\btn}{\bar{\theta}^n}

\newcommand{\lp}{\left(}
\newcommand{\rp}{\right)}
\newcommand{\lc}{\left[}
\newcommand{\rc}{\right]}
\newcommand{\lcl}{\left\{}
\newcommand{\rcl}{\right\}}
\newcommand{\lln}{\left|}
\newcommand{\rrn}{\right|}

\newcommand{\lbp}{\Big(}
\newcommand{\rbp}{\Big)}
\newcommand{\lbc}{\Big[}
\newcommand{\rbc}{\Big]}
\newcommand{\rbcl}{\Big\}}
\newcommand{\lbcl}{\Big\{}

\newcommand{\bean}{\begin{eqnarray*}}
\newcommand{\eean}{\end{eqnarray*}}
\newcommand{\ben}{\begin{enumerate}}
\newcommand{\een}{\end{enumerate}}
\newcommand{\beq}{\begin{equation}}
\newcommand{\eeq}{\end{equation}}

\newcommand{\bepp}{\mathcal B_{\eta,2p}}

\newcommand{\absg}[1]{\big|#1\big|}
\newcommand{\normg}[1]{\big\|#1\big\|}

\newtheorem{theorem}{Theorem}[section]

\newtheorem{corollary}[theorem]{Corollary}

\newtheorem{notation}[theorem]{Notation}

\newtheorem{hypothesis}{Hypothesis}
\newtheorem{lemma}[theorem]{Lemma}

\newtheorem{proposition}[theorem]{Proposition}

\theoremstyle{remark}
\newtheorem{remark}[theorem]{Remark}

\begin{document}

\title[Stratonovich heat equation]{The Stratonovich heat equation: a continuity result and weak approximations}

\author{Aurélien Deya, Maria Jolis \and Lluís Quer-Sardanyons}

\address{
{\it Aurélien Deya:}
{\rm Institut \'Elie Cartan Nancy, B.P. 239,
54506 Vandoeuvre-l\`es-Nancy Cedex, France}.
{\it Email: }{\tt Aurelien.Deya@iecn.u-nancy.fr}
\newline
$\mbox{ }$\hspace{0.1cm}
{\it Maria Jolis \and Lluís Quer-Sardanyons:}
{\rm Departament de Matem\`atiques, Facultat de Ci\`encies, Edifici C, Universitat Aut\`onoma de Barcelona, 08193 Bellaterra, Spain}.
{\it Email: }{\tt quer@mat.uab.cat, mjolis@mat.uab.cat}
}

\keywords{convergence in law; stochastic heat equation; Stratonovich integral; convolutional rough paths theory}

\subjclass[2010]{60H10, 60H05, 60H07}

\date{\today}

%\thanks{}

\begin{abstract}
We consider a Stratonovich heat equation in $(0,1)$ with a nonlinear multiplicative noise driven by a trace-class Wiener process. First, the equation is shown to have a unique mild solution. Secondly, convolutional rough paths techniques are used to provide an almost sure 
continuity result for the solution with respect to the solution of the 'smooth' equation obtained by replacing the noise with an absolutely continuous process. This continuity result is then exploited to prove weak convergence results based on Donsker and Kac-Stroock type approximations of the noise.
\end{abstract}

\maketitle

\section{Introduction and main results}
\label{sec:intro}

The main motivation of the paper comes from \cite{Bardina-Jolis-Quer}, where the authors consider, for some fixed $T>0$, the stochastic heat equation
\beq
 \frac{\partial Y^n}{\partial t}(t,x)-\frac{\partial^2 Y^n}{\partial x^2}(t,x) = \dot \theta^n(t,x), \quad (t,x)\in [0,T]\times [0,1],
\label{eq:67}
\eeq
with some initial data and Dirichlet boundary conditions, where the random fields
$(\dot \theta^n)_{n\geq 1}$ verify that the family of processes $\theta^n(t,x):=\int_0^t\int_0^x \dot \theta^n(s,y)\, dyds$
converge {\it{in law}}, in the space $\mathcal{C}([0,T]\times [0,1])$ of continuous functions, to the Brownian sheet.
Then, sufficient conditions on $\theta^n$ are provided 
such that $Y^n$ converges {\it{in law}}, as $n\rightarrow \infty$, to the mild solution $Y$ of 
\[
 \frac{\partial Y}{\partial t}(t,x)-\frac{\partial^2 Y}{\partial x^2}(t,x) = \dot{W}(t,x), \quad (t,x)\in [0,T]\times [0,1],
\]
where $\dot{W}(t,x)$ stands for the space-time white noise. Applications of this result include the case of a Donsker type approximation, as well as a Kac-Stroock type approximation in the plane. 

\smallskip

Such {\it{diffusion approximation}} issues for stochastic PDEs have been extensively studied in the literature. 
Let us quote here Walsh \cite{Walsh-AAP}, Manthey \cite{Manthey1,Manthey2}, Tindel \cite{Tindel-elliptic}, Carmona and Fouque \cite{Carmona-Fouque}, 
Florit and Nualart \cite{Florit-Nualart}, just to mention but a few.

\smallskip

Now, following the line of \cite{Bardina-Jolis-Quer}, a natural question to be dealt with is to try to get the same type of weak convergence in a non-additive situation, that is when the term $\dot \theta^n(t,x)$ in (\ref{eq:67})
is replaced with $f(Y^n(t,x)) \dot \theta^n(t,x)$, for some sufficiently smooth function $f:\R \to \R$. In this case, one expects that the limit 
equation is of Stratonovich type, as it was the case in \cite{Carmona-Fouque} and \cite{Florit-Nualart} 
(see also \cite{Bally-Millet-Sanz,Tessitore-Zabczyk-JEE2006} for examples of a similar behaviour). This phenomenon has been recently illustrated by 
Bal in \cite{Bal-CMP2009} as well, for the weak approximation of a linear
parabolic equation in $\RR^d$ with random potential given by $Y^n(t,x) \dot \theta^n(x)$.

\smallskip

Going back to our setting, and focusing first on what we expect to be our limit equation, we should consider the Stratonovich heat equation 
\beq
 \frac{\partial Y}{\partial t}(t,x)-\frac{\partial^2 Y}{\partial x^2}(t,x) = f(Y(t,x)) \circ \dot{W}(t,x), \quad (t,x)\in [0,T]\times [0,1].
\label{eq:68}
\eeq
Unfortunately, a well-known drawback in this situation is that the solution admits only very low regularity (see \cite{Tindel-SSR1997}), a major obstacle for our treatment of the non-linearity of the problem. For this reason, we have chosen to restrict our attention to the case of a trace-class noise. To be more specific, we will assume that $\dot{W}$ is the formal derivative of a $L^2(0,1)$-valued Wiener process $\{W_t,\, t\in [0,T]\}$ with covariance operator $Q$ satisfying the following property:
\begin{hypothesis}\label{hypo-noise-3}
Let $(e_k)_{k\geq 1}$ be the basis of eigenfunctions for the Dirichlet Laplacian $\Delta$ in $L^2(0,1)$ 
given by $e_k(x):=\sqrt{2} \sin(k\pi x)$, $x\in [0,1]$. We assume that there exists a sequence of non-negative real numbers $(\la_k)_{k\geq 1}$ and a parameter $\eta >0$ such that $Q e_k=\la_k e_k$ for 
every $k\geq 1$ and $\sum_{k\geq 1} (\la_k  k^{4\eta}) < \infty$. Without loss of generality, we assume that $\eta \in (0,\frac18)$.
\end{hypothesis}
In particular, for any fixed $t\geq 0$, the process $W_t$  can be expanded in $L^2(\Omega;L^2(0,1))$ as
\begin{equation}\label{repr-sum}
W_t=\sum_{k\geq 1} \sqrt{\la_k} \be^k_t\,  e_k,
\end{equation}
where $(\be^k)_{k\geq 1}$ is a family of independent Brownian motions. Note that the condition $\sum_{k\geq 1} (\la_k \cdot k^{4\eta}) < \infty$ is only 
slightly stronger than the usual trace-class hypothesis $\sum_{k\geq 1}\la_k < \infty$, insofar as $\eta$ can be chosen as small as one wishes. 
For instance, it covers the case where $Q=(\mbox{Id} - \Delta)^{-r}$ with $r>\frac12$.
%$\la_k=k^{-r}$ for any $r>1/2$, since we can then pick $\eta$ such that $2r-4\eta >1$.

\smallskip

Another change with respect to \cite{Bardina-Jolis-Quer} lies in our formulation of the study:
compared to the random field approach in \cite{Bardina-Jolis-Quer}, here it has turned out to be more convenient to 
use the Hilbert-space-valued setting of Da Prato and Zabczyk \cite{DaPrato-Zabczyk}. In particular, 
we are interested in the mild form of equation (\ref{eq:68}),  which is given by
\begin{equation}\label{equa-mild}
Y_t=S_t\psi +\int_0^t S_{t-u}(f(Y_u) \circ dW_u ), \qquad t\in [0,T],
\end{equation}
where from now on, we will use the notation $Y_t(\cdot):=Y(t,\cdot)$,  
$\psi$ is some initial condition and $f:\R \to \R$ is a smooth enough mapping. As usual, $(S_t)_{t>0}$ denotes the strongly continuous 
semigroup of operators generated by $-\Delta$.

%A well-known drawback in this situation is that such an equation does not have a square integrable solution (see \cite{Tindel-SSR1997}).   
%In order to overcome this problem, we have chosen to work with a more regular noise in space. Thus, setting $\cb:=L^2((0,1))$, 
%we will assume that $\dot{W}$ is the formal derivative of a $\cb$-valued Wiener process $\{W_t,\, t\in [0,T]\}$ with covariance 
%operator $Q$ of nuclear type, that is $\mbox{Tr } Q <\infty$ (see Hypothesis \ref{hypo-noise-3} for further details). Another change with respect to \cite{Bardina-Jolis-Quer} lies in the formulation of our study:

%Here, the notation $\cdot$ stands for a pointwise multiplication of functions, i.e., $(\vp \cdot \psi)(\xi):=\vp(\xi) \psi(\xi)$.

\

A first part of the paper (Section \ref{sec:strato}) will be devoted to the interpretation of (\ref{equa-mild}) as a {\it{Stratonovich}} equation, 
and it will allow us to exhibit an existence and uniqueness result for the solution. 
We should mention here that the stochastic heat equation in the Stratonovich framework has already been studied in various settings,
most of them in the case of a linear multiplicative noise (see e.g. \cite{Buckdahn-Ma-I,Buckdahn-Ma-II,Hu-Nualart-PTRF2009}).
Once we have given a full sense to (\ref{equa-mild}), our strategy to study weak approximations of the solution could be stated in the following loose form:

\smallskip

\noindent {\it{(a)}} We will first establish an {\it{almost sure}} continuity result (in some suitable space-time topology) for the solution of (\ref{equa-mild}) with respect to the solution of the 'smooth' heat equation obtained by replacing $W$ with an absolutely continuous process $\widetilde{W}$ (see Theorem \ref{thm:continuity}).

\noindent {\it{(b)}} Then, for two particular families of absolutely continuous processes approximating $W$, we will rely on our continuity result to show convergence towards the solution in some possibly different probability space, leading us to the expected {\it{weak}} convergence (see Theorem \ref{theo:approx}).    

\

Our strategy to compare the solution $Y$ of (\ref{equa-mild}) with 'smooth' solutions is based on a genuine {\it{rough-paths}} type expansion of the equation, which follows the ideas of \cite{GT,RHE,RHE-glo}. Rough-paths techniques have indeed proved to be very efficient as far as approximation of non-linear systems in finite dimension is concerned (see \cite[Chapter 17]{FVbook}), and it is therefore natural to address the same question in this infinite-dimensional background. Note that the model given by (\ref{equa-mild}) differs from those studied in \cite{RHE,RHE-glo}, where only finite-dimensional noises are considered, forcing us to revise most of the technical details behind this procedure (see Section \ref{sec:prel}).

\

In order to state the above-mentioned results with more details, we need to introduce the spaces in which our random variables will take their values. First, as far as the spatial regularity is concerned, the fractional Sobolev spaces must come into the picture. Thus, for every $\al \in \R$ and $p\geq 2$, we will denote by $\cb_{\al,p}$ the fractional Sobolev space of order $\al$ based on $L^p(0,1)$, that is
\[
\vp \in \cb_{\al,p} \Longleftrightarrow (-\Delta)^\al \vp \in L^p(0,1),
\]
where $\Delta$ stands for the Dirichlet Laplacian in $L^2(0,1)$ (see e.g. \cite{run-sick} for a thorough study of these spaces). 
For the sake of conciseness, we will write $\cb_\al$ for $\cb_{\al,2}$ and $\cb$ for $\cb_0=L^2(0,1)$ throughout the paper.
We will also denote by $\cb_\infty$ the set of continuous functions on $[0,1]$, endowed with the supremum norm.

Of course, we will also have to deal with the time regularity of our processes.
So, for any subinterval $I\subset [0,T]$ and any Banach space $V$, we define $\cac^0(I;V)$ as the space of continuous functions
$y:I\rightarrow V$ and set 
\[
 \cn[y;\cac^0(I;V)]:=\sup_{t\in I} \norm{y_t}_V .
\]
Moreover, for any $\la>0$, we introduce the space $\cac^\la(I;V)$ of $\la$-H\"older continuous $V$-valued functions endowed with 
the seminorm  
\begin{equation}\label{der-norm}
\cn[y;\cac^\la(I;V)] :=\sup_{s<t \in I} \frac{\norm{y_t-y_s}_V}{\lln t-s \rrn^\la}.
\end{equation}
Note that in the case where $I=[0,T]$, we will often write $\mathcal{C}^\la(V)$ for $\cac^\la([0,T];V)$.

\

Now, consider any process $\widetilde{W}$ defined on the same probability space as $W$ and with absolutely continuous paths in $\cb_{\eta,2p}$, for every integer $p\geq 1$ (recall that $\eta$ has been defined in Hypothesis \ref{hypo-noise-3}). 
Then, let $\{\widetilde{Y}_t,\, t\in [0,T]\}$ be the unique solution of the Riemann-Lebesgue equation (considered in a pathwise sense):
\begin{equation}\label{regu-equa}
\widetilde{Y}_t=S_t \widetilde{\psi}+\int_0^t S_{t-u}(f(\widetilde{Y}_u) \cdot d\widetilde{W}_u), \qquad t\in [0,T],
\end{equation}
where $\tilde{\psi}\in \cb$. As evoked earlier, our first main result will consist in comparing such a solution $\widetilde{Y}$ with the solution $Y$ of (\ref{equa-mild}). This result can be stated as follows.

\begin{theorem}\label{thm:continuity}
Assume that Hypothesis \ref{hypo-noise-3} holds true for $W$ and some parameter $\eta >0$, and let $f:\R \to \R$ be a 
function of class $\cac^3$, bounded with bounded derivatives. In addition, pick $\ga \in (\frac12,\frac12+\eta)$ and assume that the initial 
condition $\psi$ (resp. $\widetilde{\psi}$) in (\ref{equa-mild}) (resp. in (\ref{regu-equa})) belongs to $\cb_\ga$. 
Then there exist $\ep >0$ and $p\geq 1$ such that
\begin{multline}\label{cont-ito}
\cn[Y-\widetilde{Y};\cac^0(\cb_\ga)]
\leq F_{\ep,p}\lp \norm{\psi}_{\cb_\ga},\norm{\widetilde{\psi}}_{\cb_\ga},\cn[W;\cac^{\frac{1}{2}-\ep}(\cb_{\eta,2p})],
\cn[\widetilde{W};\cac^{\frac{1}{2}-\ep}(\cb_{\eta,2p})] \rp\\
 \lcl \norm{\psi-\widetilde{\psi}}_{\cb_\ga} +\cn[W-\widetilde{W};\cac^{\frac{1}{2}-\ep}(\cb_{\eta,2p})]\rcl,
\end{multline}
for some deterministic function $F_{\ep,p}:(\R^+)^4 \to \R^+$ bounded on bounded sets.
\end{theorem}

The topologies involved in this statement are directly inherited from our rough-paths analysis of the equation, and their relevance should therefore become clear through the lines of Section \ref{sec:prel} (see in particular the proof of the central Proposition \ref{prop:cont-op}). Note that this bound certainly remains valid with respect to some Hölder norm (in time) for the left-hand side of (\ref{cont-ito}), as our arguments will suggest it. However, due to the technicality of the rough-paths procedure, we have preferred to focus on the behaviour of the supremum norm (see also Remark \ref{rk:optim}).

\medskip

Our next step will consist in applying the above Theorem \ref{thm:continuity} - on some possibly larger probability space - to two particular families of 
absolutely continuous processes that approximate $W$, so as to retrieve weak convergence results for the solution. To define these approximation processes, 
we will make use of the following additional notation. Namely, on a probability space $(\Omega,\cf,P)$, given a sequence $(X^k)_{k\geq 1}$ of centered i.i.d processes 
admitting moments of any order, we set, for every $t\geq 0$,
\[
\mathbf{W}(X^\cdot)_t:=\sum_{k\geq 1} \sqrt{\la_k} X^k_t e_k.
\]
Thanks to our forthcoming Proposition \ref{prop:conver}, we know that $\mathbf{W}(X^\cdot)$ is indeed well-defined as a process on $(\Omega,\cf,P)$ 
with values in $\cb_{\eta,2p}$, for all $p\geq 1$. Let us also specify that, given a sequence $(\be^n)_{n\geq 1}$ of real-valued processes,
we will henceforth denote by $(\be^{n,\cdot})_{n\geq1}=(\be^{n,k})_{n,k\geq 1}$ a generic sequence of independent copies of $(\be^n)_{n\geq 1}$ 
(defined on a possibly larger probability space).

The two families of approximations at the core of our study can now be introduced as follows (we fix $T=1$ for the sake of clarity): 

\medskip

\noindent {\it{(i)}} The \emph{Donsker approximation} $W^n: =\mathbf{W}(S^{n,\cdot})$, where $S^n$ is a sequence of appropriately rescaled random walks. 
To be more specific, let $(Z_j)_{j\geq 1}$ be a family of i.i.d random variables with
mean zero, unit variance and admitting moments of any order. Then, for each $n\in\N$, set
\begin{equation}\label{dons-intro}
S_t^{n} := n^{-1/2}\lbcl \sum_{j=1}^{i-1}  Z_j+\frac{t-(i-1)/n}{1/n}\,
Z_i\rbcl\quad\text{ if }\,
t\in\lbc\frac{i-1}{n},\frac{i}{n}\rbc,\quad \text{with
}\,i\in\{1,\ldots,n\}.
\end{equation}
Recall that, by Donsker Invariance Principle (see e.g. \cite[Thm. 4.20]{Karatzas-Shreve}), $S^{n}$ is known to
converge in law to the standard Brownian motion in $\mathcal \cac^0([0,1];\R)$, as $n\to\infty$.

\medskip

\noindent {\it{(ii)}} The \emph{Kac-Stroock approximation} $W^n:=\mathbf{W}(\theta^{n,\cdot})$, where $\theta^n$ stands for the classical 
Kac-Stroock approximation of the (one-dimensional) Brownian motion. Precisely, introduce a standard Poisson process $N$ and
a Bernoulli variable $\zeta$ independent of $N$, with $P(\zeta=1)=1/2$. Then, set
\begin{equation}\label{ks-intro}
\theta_t^{n}:=\sqrt{n}\iot (-1)^{\zeta+N(ns)}ds.
\end{equation}
Here again, the sequence $(\theta^n)_{n\geq 1}$ thus defined
converges in law in $\mathcal \cac^0([0,1];\R)$, as $n\to\infty$, to a
standard Brownian motion (see e.g \cite{Kac-1956,Pinsky-PTRF1968}).

\smallskip

%We remark that, in continuity with the results of \cite{Bardina-Jolis-Quer}, we have not considered the one-dimensional counterpart 
%of the Kac-Stroock processes in the plane considered therein (see \cite{Stroock}), but the slight modification with the Bernoulli variables $\zeta$. 
%This has been done in order to guarantee that the kernels $\theta^n$ are mean zero.

Of course, the one-dimensional weak convergence of $S^n$ (resp. $\theta^n$) towards the Brownian motion is a priori not sufficient for us to 
apply Theorem \ref{thm:continuity}. Our aim is to turn this one-dimensional weak convergence into an almost sure convergence result for 
$\mathbf{W}(S^{n,\cdot})$ (resp. $\mathbf{W}(\theta^{n,\cdot})$) with respect to the topology involved in (\ref{cont-ito}), and this will appeal in particular to Skorokhod embedding arguments (see Section \ref{sec:approx-law}). Together with Theorem \ref{thm:continuity}, the strategy ends up with the following statement.

\begin{theorem}\label{theo:approx}
Under the hypotheses of Theorem \ref{thm:continuity}, fix an initial condition $\psi=\tilde{\psi} \in \cb_\ga$, and denote by $Y^n$ the (Riemann-Lebesgue) 
solution of (\ref{regu-equa}) associated with either the Donsker approximation $W^n=\mathbf{W}(S^{n,\cdot})$ or the Kac-Stroock 
approximation $W^n=\mathbf{W}(\theta^{n,\cdot})$. 
Then, as $n\to \infty$, $Y^n$ converges in law to $Y$ in the space $\cac^0(\cb_\ga)$.
\end{theorem}

\

The paper is organized as follows. Section \ref{sec:strato} is devoted to a few preliminaries on the theoretical study of the Stratonovich heat equation (\ref{eq:68}). 
The rough-paths type analysis of this equation is performed in Section \ref{sec:prel}, and it will lead us to the proof of our continuity result Theorem \ref{thm:continuity}. 
In Section \ref{sec:approx-law}, we will tackle the approximation issue for the above-defined Donsker and Kac-Stroock processes by exhibiting a general convergence criterion (see Proposition \ref{prop:crit}), which will entail Theorem \ref{theo:approx}. Eventually, we have added an appendix with material on fractional Sobolev spaces and the proof of a 
technical result needed in Section \ref{sec:proofthm}.  

\

\begin{remark}\label{rk:commut}
At first sight, the reader familiar with rough-paths type continuity results may be surprised at the absence of some `Lévy-area' term in our bound (\ref{cont-ito}). Otherwise stated, the convergence of an approximation $W^n$ towards $W$ (with respect to some appropriate topology) is sufficient to guarantee the convergence of the associated solution. In fact, on this particular point, the situation is very similar to the case of a one-dimensional SDE with so-called \emph{commuting} vector fields, i.e.,
\begin{equation}\label{sys-com}
dY_t=b(Y_t)\, dt+\sum\nolimits_{i=1}^n \si_i(Y_t) \circ dB^i_t \quad \text{with} \quad \si_i'(x)\si_j(x)=\si_j'(x)\si_i(x) \quad \text{for all} \ i,j=1,\ldots,n.
\end{equation}
It is a well-known fact (see for instance \cite{sussm}) that under this commuting assumption, the solution $Y$ of (\ref{sys-com}) appears as a continuous functional of the sole noise $B$ (that is, no need for any Lévy-area component). In a certain way, Equation (\ref{eq:68}) fits the above pattern. Indeed, for fixed $x\in (0,1)$, the noisy perturbation can be written as $\sum_{i=0}^\infty [\sqrt{\la_i}\,e_i(x)f(Y_t(x))]  \circ d\be^i_t$, i.e., we (morally) deal with $n=\infty$ and $\si_i(\cdot)=\sqrt{\la_i} e_i(x) f(\cdot)$ in (\ref{sys-com}). So, at least at this heuristic level, our continuity result (\ref{cont-ito}) becomes quite natural. In a more specific way, we will see that due to the commuting property, the Lévy-area term arising from the rough-paths analysis of (\ref{eq:68}) can be easily reduced to some continuous functional of $W$ (Lemma \ref{lem:ipp}).
\end{remark}

\begin{remark}\label{rmk:dimension}
As we shall see it in Section \ref{sec:prel}, our proof of Theorem \ref{thm:continuity} heavily relies on the properties of the fractional Sobolev spaces $\cb_{\al,p}$, which we have recalled in the appendix. Unfortunately, many of these properties become much more restrictive as soon as the underlying space dimension is larger than $2$, as illustrated by the classical Sobolev embeddings. This accounts for our choice to stick to a one-dimension heat equation. Note however that our considerations on the theoretical study of (\ref{equa-mild}) (Section \ref{sec:strato}) could be easily extended to a multidimensional setting.
\end{remark}

\begin{remark}
The results in this paper remain actually valid for any operator $A$ of the form $A=-\partial_x (a \cdot \partial_x)+c$, where $c\geq 0$ and $a:[0,1]\to \R$ is a continuously differentiable function. Indeed, as explained in \cite[Section 2.1]{RHE-glo}, such an operator $A$ also generates an analytic semigroup of contractions and one can identify the domains $\cd(A_p^\al)$ of its fractional powers with the spaces $\cb_{\al,p}$, which is sufficient to follow the lines of our reasoning.
\end{remark}

Unless otherwise stated, any constant $c$ or $C$ appearing in our computations below
is understood as a generic constant which might change from line to line without further mention.

%%%%%%%%%%%%%%%%%%%%%%%%%%%%%%%%%%%%%%%%%%%%%%%%%%%%%%%%%%%%%%%%%%%%%%%%%%%%%%%%%%%%%%%%%%%%%%%%%%%%%%%%%%%%%%%%%%%%%%%%%%
%%%%%%%%%%%%%%%%%%%%%%%%%%%%%%%%%%%%%%%%%%%%%%%%%%%%%%%%%%%%%%%%%%%%%%%%%%%%%%%%%%%%%%%%%%%%%%%%%%%%%%%%%%%%%%%%%%%%%%%%%%

\section{The Stratonovich integral}
\label{sec:strato}

Recall that we are interested in the following mild equation:
\begin{equation}\label{equa-mild-2}
Y_t=S_t\psi +\int_0^t S_{t-u}( f(Y_u) \circ dW_u), \quad t\in [0,T], \ \psi \in \cb,
\end{equation}
where $(S_t)_{t\geq 0}$ denotes the strongly continuous semigroup of operators generated by $-\Delta$ with Dirichlet boundary conditions, and $W$ is assumed to satisfy Hypothesis \ref{hypo-noise-3}.

\smallskip

The integral appearing in (\ref{equa-mild-2}) is thus understood in some \emph{Stratonovich sense}, an interpretation to be clarified in a convolutional setting, which is the main purpose of this first section. Once endowed with this interpretation, it turns out that (\ref{equa-mild-2}) reduces to a common mild Itô equation with an additional drift term, and accordingly the existence and uniqueness of $Y$ can be derived from well-known results (see Section \ref{sec:existence}).

\smallskip

Note that the following regularity assumption on $f$ will prevail throughout the section. 

\begin{hypothesis}\label{hypo-f}
The function $f:\RR \rightarrow \RR$ is bounded, of class $\mathcal{C}^2$ and with bounded derivatives.
\end{hypothesis}

\subsection{The Stratonovich integral}
\label{sec:strato-integral}

In order to interpret $\int_0^t S_{t-u}( f(Y_u) \circ dW_u)$, we restrict our attention to a particular class of processes $Y$. Namely, we assume that, on some filtered probability space $(\Omega,\mathcal{F},(\mathcal{F}_t)_{t\geq 0},P)$, 
$\{Y_t,\, t\in[0,T]\}$ is the unique $\cb$-valued mild solution of the following equation:
\begin{equation}\label{mild-ito-proc}
dY_t-\Delta Y_tdt=V^1_tdt+ V^2_tdW_t, \quad Y_0=\psi\in \cb,
\end{equation}
for some $\mathcal{F}_t$-adapted random fields $\{V^i_t,\, t\in[0,T]\}$, $i=1,2$, with continuous paths in $\cb$ (recall that $\cb:=L^2(0,1)$). Moreover, we assume that 
\beq
 \sup_{t\leq T} E[\|V^2_t\|^2_\cb] <+\infty.
\label{eq:557}
\eeq
In fact, such a process $Y$ is explicitly given (see e.g. \cite{DaPrato-Zabczyk}) by
\begin{equation}\label{mild-ito-proc-2}
Y_t=S_t \psi + \int_0^t S_{t-s} V^1_s ds + \int_0^t S_{t-s}\left( V^2_s \cdot dW_s\right).
\end{equation}
Now, a natural idea to define the integral in (\ref{equa-mild-2}) in some Stratonovich sense would be the following: 
introduce the kernel $G_{t-s}(x,y)$ of $S_{t-s}$ and, with the representation (\ref{repr-sum}) of $W$ in mind, set
$$ \left[\int_0^t S_{t-s}\left( f(Y_s) \circ dW_s\right)\right] (x) 
\ " =" \ \sum_{j=1}^\infty \sqrt{\la_j} \left( \int_0^t \langle G_{t-s}(x,*) f(Y_s),e_j\rangle_\cb \, \circ d\beta^j_s \right),
$$
where the symbol $*$ denotes the space variable and each integral $\int_0^t \langle G_{t-s}(x,*) \cdot f(Y_s),e_j\rangle_\cb \, \circ d\beta^j_s$ is interpreted in the (classical) Stratonovich sense. 
Nevertheless, it is a well-known fact that the process $Y$ defined by (\ref{mild-ito-proc-2}) is not always a $\cb$-valued semimartingale 
(in other words, $Y$ is not always a strong solution of (\ref{mild-ito-proc-2}), see e.g. \cite[Sec. 5.6]{DaPrato-Zabczyk}), 
making the definition of these integrals quite obscure at first sight.

\smallskip

To overcome this difficulty, we consider a standard semimartingale approximation of $Y$: for every $\ep >0$, let $Y^\ep$ be the unique (strong) solution of
\[
dY^\ep_t-\Delta_\ep Y^\ep_tdt=V^1_tdt+ V^2_tdW_t, \quad Y^\ep_0 = \psi,
\]
where
\[
 -\Delta_\ep:=\frac{1}{\ep}(\mbox{Id}-(\mbox{Id}-\ep \Delta)^{-1})
\]
stands for the Yosida approximation of $-\Delta$. In particular, $-\Delta_\ep$ defines a monotone and bounded operator which converges pointwise to $-\Delta$ (see e.g. \cite{Brezis-maxmon}). Then, for every fixed $\ep >0$, $Y^\ep$ is a semimartingale, and we have (see e.g. \cite[Proposition 7.5]{DaPrato-Zabczyk})
\begin{equation}
\lim_{\ep\rightarrow 0} \, \sup_{t\leq T} E \big[ \|Y^\ep_t-Y_t\|^2_{\cb}\big] =0.
\label{eq:556}
\end{equation}
This extrinsic procedure will lead us to the following interpretation:

\begin{proposition}\label{prop:defi-strato}
With the above notations, the family of Stratonovich integrals defined for all $t\in [0,T], x\in (0,1)$ by
\begin{equation}\label{strato-int-regu}
\left[\int_0^t S_{t-s}\left( f(Y^\ep_s) \circ dW_s\right)\right] (x) 
:= \lim_{u\nearrow t}  \sum_{j=1}^\infty \sqrt{\la_j}  \left( \int_0^u \langle G_{t-s}(x,*) \cdot f(Y^\ep_s),e_j\rangle_\cb \, \circ d\beta^j_s \right),
\end{equation}
where the latter limit is considered in $L^2(\Omega)$, converges in $L^1(\Omega; \cac^0([0,T];\cb))$ as $\ep$ tends to $0$. Its limit, that we denote by $\int_0^\cdot S_{\cdot-s}\left( f(Y_s) \circ dW_s\right)$, satisfies the relation
\begin{equation}\label{decompo-strato}
\int_0^t S_{t-s}\left( f(Y_s) \circ dW_s\right)=
\int_0^t S_{t-s}\left( f(Y_s) \cdot dW_s\right) + \int_0^t S_{t-s}\left( V^2_s \cdot f'(Y_s)\cdot  P\right)  ds,
\end{equation}
where $P(\xi):=\frac{1}{2}\sum_{k=1}^\infty \la_k e_k(\xi)^2$ and the notation $\int_0^t S_{t-s}\left( f(Y_s) \cdot dW_s\right)$ refers to the (usual) Itô integral.
\end{proposition}

\

Thus, the Stratonovich integral in (\ref{equa-mild-2}) will henceforth be understood as in the latter proposition, in the class of processes $Y$ 
satisfying an equation of the form (\ref{mild-ito-proc}). Note that the relation (\ref{decompo-strato}) provides us with a familiar decomposition 
for the Stratonovich integral as the sum of an Itô integral and a trace term, and it must be compared with the decomposition for the (standard) Stratonovich integral.

\smallskip

As a first step in the proof of Proposition \ref{prop:defi-strato}, observe that the two terms in the right-hand side of (\ref{decompo-strato}) are indeed well-defined processes in $L^2(\Omega;\cb)$. This is a straightforward consequence of the boundedness of $f,f'$, the trace-class assumption on $W$, and the fact that $P$ defines a uniformly bounded function.

\smallskip

We point out that, in the definition (\ref{strato-int-regu}), we first restrict the integral to $(0,u)$ with $u<t$ in order to avoid the singularity 
in the derivative of the kernel $G$.
This will be clarified in the proof of the next lemma. 

\begin{lemma}
With the above notations, we have that, for all $u\in (0,t)$,
\begin{align}\label{eq:554}
& \sum_{j=1}^\infty \sqrt{\la_j}  \left( \int_0^u \langle G_{t-s}(x,*) \cdot f(Y^\ep_s),e_j\rangle_\cb \, \circ d\beta^j_s \right) \\
& \qquad = \left[ \int_0^u S_{t-s}\left( f(Y^\ep_s) \cdot dW_s\right)\right] (x) 
+  \int_0^u [S_{t-s}( f'(Y^\ep_s) \cdot V^2_s\cdot P)] (x)\,  ds. \nonumber 
\end{align}
\end{lemma}

\begin{proof}
For any fixed $(u,x)\in (0,t)\times (0,1)$ and $j\in \mathbb{N}$, the process $s\mapsto \langle G_{t-s}(x,*) \cdot f(Y^\ep_s),e_j\rangle_\cb$, $s\in [0,u]$,
defines a (real-valued) semimartingale. Hence, we can use Itô's formula to assert that
\begin{align}
 \langle G_{t-s}(x,*) \cdot f(Y^\ep_s),e_j\rangle_\cb & = \langle G_t(x,*)\cdot f(\psi),e_j\rangle_\cb 
+ \int_0^s \langle \partial_t G_{t-r}(x,*) \cdot f(Y^\ep_r), e_j\rangle_\cb \, dr \nonumber \\
& \quad + \int_0^s \langle G_{t-r}(x,*)\big\{ (\Delta_\ep Y^\ep_r + V^1_r)\cdot  f'(Y^\ep_r) + \frac12 V^2_r\cdot  f''(Y^\ep_r)\big\}, e_j\rangle_\cb \, dr \nonumber \\
& \quad + \sum_{k=1}^\infty \sqrt{\la_k} \int_0^s \langle G_{t-r}(x,*)\cdot  f'(Y^\ep_r) \cdot V^2_r\cdot  e_k,e_j\rangle_\cb  \, d\beta^k_r. 
\label{eq:555}
\end{align}
The hypotheses on $f$ and $V^i$, and the fact that $s\leq u<t$, guarantee that all terms on the right-hand side above are well-defined. More precisely, 
using the spectral decomposition of $G$ given by
\[
 G_{t-r}(x,y)= \sum_{k=1}^\infty e^{-k^2 \pi^2 (t-r)} e_k(x) e_k(y),
\]
one proves that, $P$-a.s., 
\[
 \int_0^s \langle \partial_t G_{t-r}(x,*) \cdot f(Y^\ep_r), e_j\rangle_\cb \, dr \leq C\,  \sum_{k=1}^\infty \Big( e^{-k^2 \pi^2 (t-s)} - e^{-k^2 \pi^2 t}\Big),
\]
and the latter is finite since $s\in (0,t)$. As far as the second pathwise integral on the right-hand side of (\ref{eq:555}) is concerned, we have, for instance, 
\begin{align*}
 E \left[ \left| \int_0^s \langle G_{t-r}(x,*)\cdot \Delta_\ep Y^\ep_r \cdot f'(Y^\ep_r), e_j\rangle_\cb \, dr\right|^2 \right] & 
 \leq C\, E \left[ \int_0^s \int_0^1 G_{t-r}(x,y) |[\Delta_\ep Y^\ep_r] (y)|^2 \, dydr \right] \\
& \leq C\, \Big( \sup_{r\leq T} E[\|Y^\ep_r\|^2_\cb] \Big) \int_0^s \frac{1}{\sqrt{t-r}} dr < +\infty.
\end{align*}
Here, we have used the fact that $0\leq G_{t-r}(x,y) \leq (2\pi(t-r))^{-1/2} \, e^{-\frac{(x-y)^2}{2(t-r)}}$.
Similarly, one easily proves that the last term in (\ref{eq:555}) is a well-defined square-integrable random variable. 

\smallskip

Plugging the expression (\ref{eq:555}) in (\ref{eq:554}) and using the definition of the (standard) Stratonovich integral 
(see e.g. formula (3.9) in \cite[p. 156]{Karatzas-Shreve}), we end up with   
\begin{align*}
 & \sum_{j=1}^\infty \sqrt{\la_j}  \left( \int_0^u \langle G_{t-s}(x,*) \cdot f(Y^\ep_s),e_j\rangle_\cb \, \circ d\beta^j_s \right)\\
& \quad = 
\sum_{j=1}^\infty \sqrt{\la_j}  \left( \int_0^u \langle G_{t-s}(x,*) \cdot f(Y^\ep_s),e_j\rangle_\cb \, d\beta^j_s \right) \\
& \qquad \qquad  + \frac12 \sum_{j=1}^\infty \la_j \int_0^u \langle G_{t-s}(x,*)\cdot  f'(Y^\ep_s) \cdot V^2_s\cdot  e_j,e_j\rangle_\cb  \, ds \\
& \quad = \left[ \int_0^u S_{t-s}\left( f(Y^\ep_s) \cdot dW_s\right)\right] (x) + 
 \int_0^u \left[ S_{t-s}( f'(Y^\ep_s) \cdot V^2_s\cdot P)\right] (x)\, ds,
\end{align*}
which concludes the proof.
\end{proof}

We can now go back to our main statement.

\begin{proof}[Proof of Proposition \ref{prop:defi-strato}]
First, owing to the previous lemma, we have that the limit on the right-hand side of (\ref{strato-int-regu}) equals to  
\[
 \left[ \int_0^t S_{t-s}\left( f(Y^\ep_s) \cdot dW_s\right)\right] (x) 
+  \int_0^t [S_{t-s}( f'(Y^\ep_s) \cdot V^2_s\cdot P)] (x)\,  ds.
\]
This can be proved using the bounded convergence theorem. Hence,
the proof reduces to the two assertions:
\begin{equation}\label{eq:334}
 \lim_{\ep\rightarrow 0} \, \sup_{t\leq T} E \bigg[\left\| \int_0^t S_{t-s}\left( f(Y^\ep_s) \cdot dW_s\right) - 
 \int_0^t S_{t-s}\left( f(Y_s) \cdot dW_s\right)\right\|_{\cb}\bigg] =0.
\end{equation}
and 
\begin{equation}\label{eq:335}
 \lim_{\ep\rightarrow 0} \, \sup_{t\leq T} E \bigg[ \left\| \int_0^t S_{t-s} ( f'(Y^\ep_s) \cdot V^2_s\cdot P) ds - 
 \int_0^t S_{t-s} ( f'(Y_s) \cdot V^2_s\cdot P) ds \right\|_{\cb}\bigg] =0.
\end{equation}
Let us first deal with (\ref{eq:334}). By the isometry property of the stochastic integral,
the boundedness of $S_{t-s}$ and the assumptions on $f$, we have:
\begin{align*}
 & E\bigg[ \left\| \int_0^t S_{t-s}\left( f(Y^\ep_s) \cdot dW_s\right) - 
 \int_0^t S_{t-s}\left( f(Y_s) \cdot dW_s\right)\right\|^2_{\cb} \bigg]\\
& \quad = E\bigg[ \int_0^t \sum_{k=1}^\infty  \lambda_k \|S_{t-s}( [f(Y^\ep_s)-f(Y_s)] e_k )\|_{\cb}^2\, ds\bigg] \\
& \quad \leq \sum_{k=1}^\infty \lambda_k \, E\bigg[ \int_0^T \int_0^1 |f(Y^\ep(s,y))-f(Y(s,y))|^2 |e_k(y)|^2\, dy ds\bigg] \\
& \quad \leq C\sum_{k=1}^\infty \lambda_k \, \sup_{t\leq T} E \big[ \|Y^\ep_t-Y_t\|_{\cb}^2\big]  \leq C\, \sup_{t\leq T} E \big[ \|Y^\ep_t-Y_t\|_{\cb}^2\big],
\end{align*}
upon recalling that $\sum_{k=1}^\infty \lambda_k < \infty$. 

In order to prove (\ref{eq:335}), we use the Sobolev embedding $L^1(0,1) \subset \cb_{-\frac14-\ep}$
and the assumptions on $f$ and $V^2$. In fact, we have
\begin{eqnarray}
\lefteqn{E\bigg[ \Big\|\int_0^t S_{t-s}\big( [f'(Y^\ep_s)-f'(Y_s)] \cdot V^2_s \cdot P\big) \, ds\Big\|_\cb \bigg]}\label{l-1-conv}\\
&\leq & C\, \|P\|_{\cb_\infty}\int_0^t \lln t-s \rrn^{-\frac14-\ep} E\big[ \| [f'(Y^\ep_s)-f'(Y_s)] \cdot V^2_s\|_{L^1}\big] \, ds\nonumber\\
&\leq & C\, \|P\|_{\cb_\infty}\int_0^t \lln t-s \rrn^{-\frac14-\ep} E\big[ \| [f'(Y^\ep_s)-f'(Y_s)]\|_\cb \| V^2_s\|_{\cb}\big] \, ds\nonumber\\
&\leq & C\, \Big( \sup_{t\leq T} E[\|Y^\ep_t-Y_t\|_\cb^2]^{1/2}\Big)\Big( \sup_{t\leq T} E[\|V^2_t\|_\cb^2]^{1/2}\Big) \nonumber
\end{eqnarray}
(recall that $\|\cdot\|_{\cb_\infty}$ refers to the supremum norm on $[0,1]$). Therefore, by the assumptions on $V^2$, the convergence (\ref{eq:556}) guarantees that (\ref{eq:334}) and (\ref{eq:335}) hold, and this lets us conclude the proof.
\end{proof}

%%%%%%%%%%%%%%%%%%%%%%%%%%%%%%%%%%%%%%%%%%%%%%%%%%%%%%%%%%%%%%%%%%%%%%%%%%%%%%%

\subsection{Existence and uniqueness of solution}
\label{sec:existence}

With the notations of the previous section, consider the following mild (Itô) equation:
\begin{equation}
Y_t=S_t \psi + \int_0^t S_{t-s}\left( f(Y_s) \cdot dW_s\right) + \int_0^t S_{t-s}( f'(Y_s) \cdot f(Y_s)\cdot P) \, ds,
\label{eq:336} 
\end{equation}
where we recall that $P(\xi)= \sum_{k=1}^\infty \lambda_k e_k (\xi)^2$. Hypotheses \ref{hypo-noise-3} and \ref{hypo-f} allow us to apply standard methods and guarantee that this equation admits a unique $L^2(\Omega;\cb)$-valued solution $Y$ (see \cite{DaPrato-Zabczyk}). 
In particular, we observe that $Y$ solves an equation of the form (\ref{mild-ito-proc})  
with $V^1_t= f'(Y_t) \cdot f(Y_t)\cdot P$ and $V^2_t=f(Y_t)$ and these random fields fulfill the assumptions specified in the previous 
section. Thus, for all $t\in [0,T]$, we can define the Stratonovich integral $\int_0^t S_{t-s}\left( f(Y_s) \circ dW_s\right)$ through Proposition \ref{prop:defi-strato} and we know that
\[
 \int_0^t S_{t-s}\left( f(Y_s) \circ dW_s\right) = \int_0^t S_{t-s}\left( f(Y_s) \cdot dW_s\right) + \int_0^t S_{t-s}( f'(Y_s) \cdot f(Y_s)\cdot P) ds,
\]
which yields that $Y$ is also a solution of (\ref{equa-mild-2}). 

\smallskip

Conversely, due to (\ref{decompo-strato}), it is readily checked that any solution of (\ref{equa-mild-2}) in the class of processes satisfying an equation of the form (\ref{mild-ito-proc}) is also a solution of (\ref{eq:336}) (use the uniqueness of $V^1,V^2$ in (\ref{mild-ito-proc-2})). This provides us with the following existence and uniqueness result.

\begin{theorem}\label{thm:wp}
Assume that Hypotheses \ref{hypo-noise-3} and \ref{hypo-f} are both satisfied and that $\psi\in \cb$. Then, there exists a unique $\cb$-valued 
process $\{Y_t,\, t\in [0,T]\}$ which solves
\[
Y_t=S_t\psi +\int_0^t S_{t-u}( f(Y_u) \circ dW_u), \quad t\in [0,T].
\]
Moreover, $Y$ has a version with continuous paths and it holds that $\sup_{t\leq T} E[\|Y_t\|^2_\cb] <\infty$.
\end{theorem}

%%%%%%%%%%%%%%%%%%%%%%%%%%%%%%%%%%%%%%%%%%%%%%%%%%%%%%%%%%%%%%%%%%%%%%%%%%%%%%%%%%%%%%%%%%%%%%%%%%%%%%%%%%%%%%%%%%%%%%%%%%%%
%%%%%%%%%%%%%%%%%%%%%%%%%%%%%%%%%%%%%%%%%%%%%%%%%%%%%%%%%%%%%%%%%%%%%%%%%%%%%%%%%%%%%%%%%%%%%%%%%%%%%%%%%%%%%%%%%%%%%%%%%%%%

\section{A rough-paths type analysis of the equation}
\label{sec:prel}

Let us now turn to the proof of Theorem \ref{thm:continuity}. As announced in the Introduction, our strategy is based on a rough-paths type expansion of the equation. Accordingly, a few ingredients taken from the so-called \emph{convolutional rough paths theory}, that is rough paths theory adapted to mild evolution equation, must be introduced in the first place.

%Next step will exploit the rough path machinery in order to prove that the increments of $Y$ can be suitably represented in terms of a kind of iterated integral with respect to $W$.
%This will be done in Section \ref{sec:representation}. Finally, such an expansion will be used in Section \ref{sec:continuity} to obtain our continuity
%result for equation (\ref{equa-mild-2}).

%The first part of this section is devoted to introduce the fractional Sobolev spaces in which our processes will live and recall
%the main Sobolev embedding results needed in the sequel (see e.g \cite{run-sick}). Secondly, we will present the main techniques of
%the rough path theory which will be used to obtain our {\it{continuity}} result, including the definition of several types of Hölder spaces of functions.
%For a complete treatement of all these tools, we refer the reader to \cite{GT,RHE}.

\subsection{Tools from (convolutional) rough paths theory}
\label{subsec:incre}

We gather here some preliminary material borrowed from \cite{GT} (see also \cite{RHE,RHE-glo}). As underlined in the latter references, a key point towards a fruitful pathwise analysis of (\ref{equa-mild}) lies in the following elementary observation: due to the semigroup property $S_{t+t'}=S_t \cdot S_{t'}$, it holds that, for any $s<t$,
\bean
Y_t-Y_s&=&\int_s^t S_{t-u}(f(Y_u) \circ dW_u)+\int_0^s \big[S_{t-u}-S_{s-u} \big](f(Y_u) \circ dW_u)\\
&=&\int_s^t S_{t-u} \lp     f(Y_u) \circ dW_u\rp+a_{ts} Y_s , \quad \text{where} \quad a_{ts}:=S_{t-s}-\id.
\eean
Otherwise stated, by setting $(\delha Y)_{ts}:=(Y_t-Y_s)-a_{ts} Y_s$, the equation (\ref{equa-mild}) can be equivalently written in the convenient form:
\begin{equation}\label{eq-base}
Y_0=\psi, \quad \quad (\delha Y)_{ts}=\int_s^t S_{t-u} \lp  f(Y_u)\circ dW_u \rp, \quad 0\leq s\leq t\leq 1.
\end{equation}
This should be compared with the behaviour of solutions to standard (stochastic) differential equations: if $X_t=a+\int_0^t \si(X_u) \, dB_u$, then $(\der X)_{ts}:=X_t-X_s=\int_s^t \si(X_u) \, dB_u$. Then, in a rough-paths setting, we are naturally led to extend the definition of $\delha$ to processes with $2$ variables, as follows:

\begin{notation}
For all processes $y:[0,T] \to \cb$ and $z:\cs_2\to \cb$, where $\cs_2:=\{(s,t)\in [0,T]^2: \ s\leq t\}$ denotes the two-dimensional simplex,
we set, for $s\leq u\leq t \in [0,T]$:
\begin{equation}\label{incr-std}
(\der y)_{ts}:=y_t-y_s, \quad  \quad (\delha y)_{ts}:=(\der y)_{ts}-a_{ts}y_s=y_t-S_{t-s}y_s,
\end{equation}
\begin{equation}\label{incr-twis}
(\delha z)_{tus}:=z_{ts}-z_{tu}-S_{t-u}z_{us}.
\end{equation}
\end{notation}

To make the notations (\ref{incr-std})-(\ref{incr-twis}) even more legitimate in this convolutional context,
let us point out the following algebraic properties:
\begin{proposition}\label{prop:telesc}
For any $y:[0,T]\to \cb$, it holds:
\begin{itemize}
\item[(i)] Telescopic sum: $\delha (\delha y)_{tus}=0$ and $(\delha y)_{ts}=\sum_{i=0}^{n-1} S_{t-t_{i+1}}(\delha y)_{t_{i+1}t_i}$ for any partition $\{s=t_0 <t_1 <\ldots <t_n=t\}$ of an interval $[s,t]$ of $[0,T]$.
\item[(ii)] Chasles relation: if $\cj_{ts}:=\int_s^t S_{t-u} \lp  y_u \cdot  dW_u\rp$, then $\delha \cj=0$.
\end{itemize}
\end{proposition}
Both points (i) and (ii) are straighforward consequences of the semigroup property. Now, in accordance with the new expression (\ref{eq-base}) for the equation, a $\delha$-version of the classical Hölder norm must come into the picture. To this end, fix a subinterval $I\subset [0,T]$ and a Banach space $V$. Then, if $y:I\to V$ and $\la>0$, set
\begin{equation}\label{delha-norm}
\cn[y;\cacha^\la(I;V)] :=\sup_{s<t \in I} \frac{\norm{(\delha y)_{ts}}_V}{\lln t-s \rrn^\la},
\end{equation}
and define $\cacha^\la(I;V)$ as the set of processes $y:I\to V$ such that $\cn[y;\cacha^\la(I;V)] <\infty$.

\smallskip

As we will see it in the sequel, a proper control for the expansion of $\int_s^t S_{t-u}(f(Y_u) \circ dW_u)$ also requires the extension of both definitions (\ref{der-norm}) and (\ref{delha-norm}) to processes with 2 or 3 variables. Precisely, if $z:\cs_2 \to V$ and $h:\cs_3 \to V$, where $\cs_3:=\{(t,u,s) \in [0,T]^3: \ s\leq u\leq t \}$, we set
\begin{equation}
\cn[z;\cac_2^\la(I;V)]:=\sup_{ s <t\in I} \frac{\norm{z_{ts}}_V}{\lln t-s \rrn^\la}, \quad
\quad \cn[h;\cac_3^\la(I;V)]:=\sup_{ s <u<t\in I} \frac{\norm{h_{tus}}_V}{\lln t-s \rrn^\la},
\end{equation}
and we define $\cac_2^\la(I;V)$ (resp. $\cac_3^\la(I;V)$) along the same lines as $\cacha^\la(I;V)$. Observe for instance that if $y\in \cac_2^\la(I;\cl(V,W))$ and $z\in \cac_2^\be(I;V)$, then
the process $h$ defined as $h_{tus}=y_{tu}z_{us}$ ($s\leq u\leq t\in I$) belongs to $\cac_3^{\la+\be}(I;W)$.

\smallskip

Note that when $I=[0,T]$, we will more simply write $\cac_k^\la(V):=\cac_k^\la(I;V)$ for $k\in \{1,2,3\}$. Besides, from now on, we use the following  convenient notation for products of processes.
\begin{notation}\label{convention-indices}
If $g:\cs_n \to \cl(V,W)$ and $h:\cs_m \to W$ (with $n,m\in\{1,2,3\}$), we define the product $gh: \cs_{n+m-1} \to W$ by the formula
\begin{equation}\label{conve}
(gh)_{t_1 \ldots t_{m+n-1}}:=g_{t_1 \ldots t_n}h_{t_n \ldots t_{n+m-1}}.
\end{equation}
\end{notation}

With this convention, it is readily checked that if $g: \cs_2 \to \cl(\cb_\ka,\cb_\al)$ and $h:[0,T] \to \cb_\ka$, then $\delha(gh): \cs_{3} \to\cb_\al$ obeys the rule:
\begin{equation}\label{rel-alg-prod}
\delha (gh)=(\delha g)h-g(\der h).
\end{equation}

\smallskip

To end up with this toolbox, let us report what may be seen as the cornerstone result of the convolutional rough paths theory, namely the existence of (some kind of) an inverse operator for $\delha$, denoted by $\Laha$, and which will play a prominent role in our forthcoming decomposition (\ref{representation}). In brief, this operator allows us to get both a nice expression and a sharp estimate for the \emph{regular} terms, i.e., the terms with Hölder regularity strictly larger than $1$, that arise from the expansion of $\int_s^t S_{t-u}(f(Y_u) \circ dW_u)$ (see in particular the proof of Lemma \ref{cor:def-z}).

\begin{theorem}\label{existence-laha}
Fix an interval $I\subset [0,T]$, a parameter $\ka \geq 0$ and let $\mu >1$. For any $h\in \cac_3^\mu(I;\cb)\cap \text{Im} \, \delha$,
there exists a unique element $$\Laha h \in \cap_{\al\in [0,\mu)} \cac_2^{\mu-\al}(I;\cb_{\al})$$ such that $\delha(\Laha h)=h$.
Moreover, $\Laha h$ satisfies the following contraction property: for all $\al\in [0,\mu)$,
\begin{equation}\label{contraction-laha}
\cn[\Laha h;\cac_2^{\mu-\al}(I;\cb_{\al})] \leq c_{\al,\mu} \, \cn[h;\cac_3^\mu(I;\cb)].
\end{equation}
\end{theorem}

\

The proof of this result can be found in \cite[Theorem 3.5]{GT}.

%%%%%%%%%%%%%%%%%%%%%%%%%%%%%%%%%%%%%%%%%%%%%%%%%%%%%%%%%%%%%%%%%%%%%%%%%%%%%%%%%%%%%%%%%%%%%%%%%%%%%%%%%%%%%%%%%%%%%

\subsection{A rough-paths type expansion of the solution}
\label{sec:representation}
We are now ready to settle our reasoning, which applies to a smooth enough vector field $f$:

\begin{hypothesis}\label{hypo-f-2}
The function $f:\R \to \R$ in (\ref{eq-base}) is of class $\mathcal{C}^3$, bounded and with bounded derivatives.
\end{hypothesis}

Our main task will actually consist in establishing the following pathwise decomposition for the solution $Y$ to (\ref{eq-base}):

\begin{theorem}\label{theo:identif}
Assume that both Hypotheses \ref{hypo-noise-3} and \ref{hypo-f-2} hold true.
Fix $\ga \in (\frac{1}{2},\frac{1}{2}+\eta)$ and assume that $\psi \in \cb_\ga$. Then the $\delha$-variations of the solution $Y$ to (\ref{eq-base}) can be expanded as
\begin{equation}\label{representation}
(\delha Y)_{ts}=\int_s^t S_{t-u}(f(Y_u) \circ dW_u)=L^W_{ts}(f(Y_s))+L^{WW}_{ts}(f(Y_s)\cdot f'(Y_s))+\Laha_{ts}\big(R^Y \big),
\end{equation}
where we have set, for all $s<u<t$,
\begin{equation}\label{defi:l-w}
L^W_{ts}(\vp):=\int_s^t S_{t-u}( \vp \cdot dW_u),
\end{equation}
\begin{equation}\label{defi:l-w-w}
L^{WW}_{ts}(\vp):=\int_s^t S_{t-u}\lp \vp \cdot  (\der W)_{us} \cdot dW_u \rp+\int_s^t S_{t-u}\lp \vp \cdot P \rp  du,
\end{equation}
and
\begin{equation}
R^Y_{tus}:=-\delha \big( L^Wf(Y)+L^{WW}\big(f(Y)\cdot f'(Y)\big)\big)_{tus}.
\end{equation}
\end{theorem}

\

The theorem must be read as follows: in the expansion of $\int_s^t S_{t-u}(f(Y_u) \circ dW_u)$, we can exhibit a \emph{main} term, namely
$$L^W_{ts}(f(Y_s))+L^{WW}_{ts}(f(Y_s)\cdot f'(Y_s)),$$
and a \emph{residual} term $\Laha_{ts}\big(R^Y \big)$ with Hölder regularity strictly larger than $1$, in the sense of Theorem \ref{existence-laha} (take $\al=0$ in (\ref{contraction-laha})). Besides, from the decomposition (\ref{representation}), we can somehow conclude that the whole dynamics induced by $W$ is "encoded" through the two (stochastic) operator-valued processes $L^W$ and $L^{WW}$. So, before we turn to the proof of (\ref{representation}), let us elaborate on the properties of these two processes.

%\begin{remark}
%In order to clarify the above statement, let us make the following observations concerning the parameters $\eta$ and $\gamma$:
%So, there are (only) two parameters to deal with in this story:
%\begin{itemize}
%\item[(i)] $\eta >0$ is a small parameter given by Lemma \ref{lem:regu-noise-as}, and it is linked to the spatial regularity of the noise.
%For convenience, and without loss of generality, we will henceforth assume that $\eta \in (0,\frac{1}{8})$.
%\item[(ii)] The parameter $\ga$ must satisfy $\frac{1}{2}<\ga < \frac{1}{2}+\eta$. This will indicate the order of the fractional Sobolev space
%in which the solution $Y$ to (\ref{equa-mild-2}) takes values, that is $Y_t \in \cb_\ga$ a.s (see Lemma \ref{lem:apriori}).
%\end{itemize}
%\end{remark}

\subsection{The couple $(L^W,L^{WW})$}

At this point, we consider $L^W_{ts}$ and $L^{WW}_{ts}$ as stochastic linear operators acting on the space of smooth functions $\vp$. The following (straightforward) relation accounts for the algebraic behaviour of the couple $(L^W,L^{WW})$: it is the convolutional analog of the classical Chen's relation between a process and its Lévy area (see \cite{gubi}).

\begin{proposition}
The processes $L^W$ and $L^{WW}$ obey the following algebraic rules: For all $s<u<t$ and all smooth function $\vp$,
\begin{equation}\label{alg-rul}
(\delha L^W)_{tus}(\vp)=0 \quad , \quad (\delha L^{WW})_{tus}(\vp)=L^W_{tu}(\vp \cdot (\der W)_{us} ).
\end{equation}
\end{proposition}

Now, it matters to identify the regularity properties of $L^W$ and $L^{WW}$ as $2$-variables processes. A first clue in this direction is given by the following (a.s.) regularity result for the noise $W$ itself.

\begin{lemma}\label{lem:regu-noise-as}
Under Hypothesis \ref{hypo-noise-3}, one has (a.s.) $W\in \cac^{\frac{1}{2}-\ep}(\cb_{\eta,2p})$ for every integer $p\geq 1$ and every small $\ep>0$.
\end{lemma}

\begin{proof}
By using our forthcoming Proposition \ref{prop:conver}, we deduce that
$$E\big[\|(\der W)_{ts}\|_{\cb_{\eta,2p}}^{2pq} \big] \leq C_{p,q} E\big[|(\der \be)_{ts}|^{2pq} \big] \leq C_{p,q} \lln t-s \rrn^{pq},$$
for all $q\geq 1$, and the result is now a straightforward consequence of the Garsia-Rodemich-Rumsey Lemma \ref{lem-grr} (take $\der^\ast=\der$ and $R=\der W$ in the latter statement).
\end{proof}

Our second ingredient towards the regularity properties of $(L^W,L^{WW})$ relies on two successive observations. First, due to their relative simplicity, the two expressions (\ref{defi:l-w})-(\ref{defi:l-w-w}) can be integrated by parts. Then, owing to the some obvious commuting properties, we can turn $L^{WW}$ into an easy-to-handle functional of $\der W$. This is what we propose to detail in the proof of the following Lemma.

\begin{lemma}\label{lem:ipp}
For every smooth function $\vp$ and all $s<t\in [0,T]$, the following formulas hold true (a.s.):
\begin{equation}\label{ipp-l-w}
L^W_{ts}(\vp)=S_{t-s}(\vp \cdot (\der W)_{ts})-\int_s^t \Delta S_{t-u}(\vp \cdot (\der W)_{tu} ) \, du,
\end{equation}
\begin{equation}\label{ipp-l-w-w}
L^{WW}_{ts}(\vp)=
\frac{1}{2} \lcl S_{t-s}(\vp \cdot (\der W)_{ts}^2 )-\int_s^t \Delta S_{t-u}\lp\vp \cdot  \lc (\der W)_{tu}^2+2(\der W)_{tu}\cdot (\der W)_{us} \rc  \rp \, du \rcl.
\end{equation}
\end{lemma}

\begin{proof}
With the expansion (\ref{repr-sum}) in mind, it is easily checked, by setting $W^N_t:=\sum_{i=1}^N \sqrt{\la_i} \be^i_t e_i$, that $L^W_{ts}(\vp)=\lim_{N\to \infty} \sum_{i=1}^N \int_s^t S_{t-u}(\vp \cdot dW^N_u)$ and
$$\int_s^t S_{t-u}(\vp \cdot (\der W)_{us} \cdot dW_u) =\lim_{N\to \infty} \int_s^t S_{t-u}(\vp \cdot (\der W^N)_{us} \cdot dW^N_u),$$
where the limits are taken in $L^2(\Omega,\cb)$. The proof then reduces to applications of Itô's formula and we only elaborate on (\ref{ipp-l-w-w}). For fixed $i,j\in \{1,\ldots,N\}$, apply Itô's formula to the (random) function $F^{i,j}_{s,t}:[s,t] \times \R \times \R \to \cb$ defined by
$$F^{i,j}_{s,t}(u,x,y):=S_{t-u}(\vp \cdot e_i \cdot e_j) \big[ (x-\be^i_s)(y-\be^j_s)-(\der \be^i)_{ts} (\der \be^j)_{ts} \big]$$
so as to deduce
\begin{multline*}
0=F^{i,j}_{s,t}(t,\be^i_t,\be^j_t)=F^{i,j}_{s,t}(s,\be^i_s,\be^j_s)-\int_s^t \Delta S_{t-u}(\vp \cdot e_i \cdot e_j)\big[ (\der \be^i)_{us}(\der \be^j)_{us}-(\der \be^i)_{ts}(\der \be^j)_{ts} \big] \, du\\
+\int_s^t S_{t-u}(\vp \cdot e_i \cdot e_j) (\der \be^i)_{us} \, d\be^j_u+\int_s^t S_{t-u}(\vp \cdot e_i \cdot e_j) (\der \be^j)_{us} \, d\be^i_u+\mathbf{1}_{\{i=j\}} \int_s^t S_{t-u}(\vp \cdot e_i \cdot e_j) \, du.
\end{multline*}
By taking the sum over $i,j$, we deduce the formula
\begin{multline*}
S_{t-s}(\vp \cdot (\der W^N)_{ts}^2)-\int_s^t \Delta S_{t-u}\big( \vp \cdot [(\der W^N)_{ts}^2-(\der W^N)_{us}^2] \big) \, du\\
=2\int_s^t S_{t-u}(\vp \cdot (\der W^N)_{us} \cdot dW^N_u)+\int_s^t S_{t-u} \big( \vp \cdot \big( \sum_{i=1}^N \la_i e_i^2 \big)\big)
\end{multline*}
and by passing to the limit (in $L^2(\Omega,\cb)$), we get
$$L^{WW}_{ts}(\vp)=\frac{1}{2} \bigg\{ S_{t-s}(\vp \cdot (\der W)_{ts}^2)+\int_s^t \Delta S_{t-u}\big( \vp \cdot [(\der W)_{ts}^2-(\der W)_{us}^2]\big) \, du\bigg\}.$$
Formula (\ref{ipp-l-w-w}) immediately follows.

\end{proof}

We are now in a position to extend both $L^W_{ts}$ and $L^{WW}_{ts}$ to larger classes of functions $\vp$ and retrieve the following (a.s.) bounds, which will be at the core of our identification procedure:

\begin{proposition}\label{prop:cont-op}
Under the hypotheses of Theorem \ref{theo:identif}, for any small $\ep >0$, there exists $\tilde{\ep}>0$ and $p\geq 1$ such that (almost surely)
\begin{multline}\label{control-l-w}
\cn[L^W;\cac_2^{\frac{1}{2}-\ep}(\cl(\cb,\cb))]+\cn[L^W;\cac_2^{(\frac{1}{2}-\ga+\eta)-\ep}(\cl(\cb_{1/2},\cb_\ga))]+\cn[L^W;\cac_2^{\frac{1}{4}-\ep}(\cl(\cb_\infty,\cb_\infty))]\\
\leq c_{\ep,\tilde{\ep},p} \, \cn[W;\cac^{\frac{1}{2}-\tilde{\ep}}(\cb_{\eta,2p})],
\end{multline}
\begin{multline}\label{control-l-w-w}
\cn[L^{WW};\cac_2^{1-\ep}(\cl(\cb,\cb))]+\cn[L^{WW};\cac_2^{(1-\ga+\eta)-\ep}(\cl(\cb_{1/2},\cb_\ga))]+\cn[L^{WW};\cac_2^{\frac{3}{4}-\ep}(\cl(\cb_\infty,\cb_\infty))]\\
\leq c_{\ep,\tilde{\ep},p} \cn[W;\cac^{\frac{1}{2}-\tilde{\ep}}(\cb_{\eta,2p})]^2,
\end{multline}
for some constant $c_{\ep,\tilde{\ep},p}$.
\end{proposition}

Note here how important the assumption $\ga < \frac12+\eta$ in Theorem \ref{theo:identif} to ensure that $(\frac12-\ga+\eta)-\ep>0$ for any small enough $\ep>0$.

\begin{proof}
In fact, thanks to the representation formulas (\ref{ipp-l-w})-(\ref{ipp-l-w-w}) and the pathwise regularity of $W$ (Lemma \ref{lem:regu-noise-as}), all of these bounds can be derived from the classical properties of the fractional Sobolev spaces (see Appendix A). For instance, owing to (\ref{sobolev-inclusion}), one has, for any $p\geq 1$ and $\al \geq \frac{1}{4p}$,
\begin{equation}\label{product}
\norm{\vp \cdot (\der W)_{ts} }_{\cb_{-\al}} \leq c_{\al,p} \norm{(\der W)_{ts}}_{L^{2p}(0,1)} \norm{\vp}_{\cb},
\end{equation}
so that for any $\tilde{\ep}$ small enough,
\bean
\norm{S_{t-s}(\vp \cdot (\der W)_{ts} )}_\cb &\leq & c_{\al,p} \lln t-s \rrn^{-\al} \norm{(\der W)_{ts}}_{L^{2p}(0,1)} \norm{\vp}_\cb \qquad \text{(use (\ref{regu-semi-1}))}\\
&\leq & c_{\al,p,\tilde{\ep}}\lln t-s \rrn^{\frac{1}{2}-\tilde{\ep}-\al}\cn[W;\cac^{\frac{1}{2}-\tilde{\ep}}(\cb_{\eta,2p})]\, \norm{\vp}_\cb.
\eean
In the same way,
\bean
\norm{\Delta S_{t-u}(\vp \cdot (\der W)_{tu} }_\cb &\leq & c_{\al,p} \lln t-s \rrn^{-1-\al} \norm{(\der W)_{ts}}_{L^{2p}(0,1)} \norm{\vp}_\cb \qquad \text{(use (\ref{regu-semi-1}))}\\
&\leq & c_{\al,p,\ep}\lln t-s \rrn^{(\frac{1}{2}-\tilde{\ep}-\al)-1} \cn[W;\cac^{\frac{1}{2}-\tilde{\ep}}(\cb_{\eta,2p})]\, \norm{\vp}_\cb.
\eean
By taking $\al$ small enough, i.e., $p$ large enough, we get the expected bound, namely
$$\cn[L^W;\cac_2^{\frac{1}{2}-\ep}(\cl(\cb,\cb))]\leq c_{\ep,\tilde{\ep},p} \cn[W;\cac^{\frac{1}{2}-\tilde{\ep}}(\cb_{\eta,2p})].$$
The other estimates for $L^W$ can be proved along the same lines. As far as $L^{WW}$ is concerned, observe for instance that if $\ep>0$ is small enough, then one has
\bean
\lefteqn{\norm{\Delta S_{t-u}(\vp\cdot (\der W)_{tu} \cdot (\der W)_{us})}_{\cb_\ga}}\\
& \leq & c \lln t-u\rrn^{(\eta-\ga)-1} \norm{\vp \cdot (\der W)_{tu} \cdot (\der W)_{us}}_{\cb_\eta}\\
&\leq & c \lln t-u\rrn^{(\eta-\ga)-1} \norm{(\der W)_{tu} \cdot (\der W)_{us}}_{\cb_\eta} \norm{\vp}_{\cb_{1/2}} \qquad \text{(use (\ref{prod-sobol-1}))}\\
&\leq & c \lln t-u\rrn^{(\eta-\ga)-1} \norm{(\der W)_{tu}}_{\cb_{\eta,4}} \norm{(\der W)_{us}}_{\cb_{\eta,4}} \norm{\vp}_{\cb_{1/2}} \qquad \text{(use (\ref{prod-sobol-2}))}\\
&\leq & c_\ep \lln t-u\rrn^{(\frac{1}{2}+\eta-\ga-\tilde{\ep})-1} \lln u-s \rrn^{\frac{1}{2}-\tilde{\ep}} \cn[W;\cac^{\frac{1}{2}-\tilde{\ep}}(\cb_{\eta,4})]^2 \norm{\vp}_{\cb_{1/2}},
\eean
which entails that
$$\cn[L^{WW};\cac_2^{(1-\ga+\eta)-\ep}(\cl(\cb_{1/2},\cb_\ga))]
\leq c_{\ep,\tilde{\ep},p} \cn[W;\cac^{\frac{1}{2}-\tilde{\ep}}(\cb_{\eta,4})]^2.$$
The (analogous) proofs for the other bounds are left to the reader.
\end{proof}

\subsection{Proof of Theorem \ref{theo:identif}}
\label{sec:proofthm}

First, we need to justify that the right-hand side of the decomposition (\ref{representation}) is well-defined. This will rely (among others) on the following a priori controls for the solution $Y$. For the sake of clarity, we have postponed the proof of this statement to Appendix B.

\begin{lemma}\label{lem:apriori}
Under the hypotheses of Theorem \ref{theo:identif}, one has (almost surely)
\begin{equation}\label{apriori2-1}
Y \in \cacha^{2\eta}(\cb_\infty) \cap \cac^0(\cb_\ga),
\end{equation}
\begin{equation}\label{apriori2-2}
K^Y :=\delha Y -L^W f(Y) \in \cac_2^{\frac{1}{2}+\eta}(\cb).
\end{equation}
\end{lemma}

Recall that according to our convention (\ref{conve}), the definition of $K^Y$ in (\ref{apriori2-2}) must be understood as $K^Y_{ts}:=(\delha Y)_{ts}-L^W_{ts}( f(Y_s)) $ for every $s<t\in [0,T]$.

\begin{lemma}\label{cor:def-z}
Under the hypotheses of Theorem \ref{theo:identif}, let $Z$ be the process given by $Z_0=\psi$ and
\begin{equation}\label{defin-z}
(\delha Z)_{ts}=L^W_{ts}(f(Y_s))+L^{WW}_{ts}(f(Y_s)\cdot f'(Y_s))+\Laha_{ts}\big(R^Y \big).
\end{equation}
Then, almost surely, $Z$ is well-defined as an element of $\cacha^{2\eta}(\cb_\infty) \cap \cac^0(\cb_\ga)$, and there exists a constant $\la >0$ such that for any subinterval $I=[\ell_1,\ell_2]\subset [0,T]$, one has
\begin{equation}\label{bound-z}
\cn[Z;\cacha^{2\eta}(I;\cb_\infty)]+\cn[Z;\cac^0(I;\cb_\ga)]
\leq \|Z_{\ell_1}\|_{\cb_\ga}+c_{W,f} |I|^\la \cn[Y;\cq(I)],
\end{equation}
where we have set
\begin{equation*}
\cn[Y;\cq(I)]:=\cn[Y;\cacha^{2\eta}(I;\cb_\infty)]+\cn[Y;\cac^0(I;\cb_\ga)]+\cn[K^Y;\cac_2^{\frac12+\eta}(I;\cb)]
\end{equation*}
\end{lemma}

\begin{proof}
First, according to Theorem \ref{existence-laha}, we need to justify that $R^Y \in \cac_3^\mu(\cb)$ for some $\mu >1$. To this end, expand $R$ using the algebraic rules (\ref{rel-alg-prod}) and (\ref{apriori2-1}), which gives
\begin{equation}\label{expansion-r}
R^Y_{tus}=L^W_{tu}N_{us}+L^{WW}_{tu} \der(f(Y)\cdot f'(Y))_{us},
\end{equation}
with $N_{us}:=\der(f(Y))_{us}-(\der W)_{us}\cdot f(Y_s)\cdot f'(Y_s)$. Thanks to (\ref{control-l-w-w}) and (\ref{apriori2-1}), it is readily checked that $L^{WW} \der(f(Y)\cdot f'(Y)) \in \cac_3^{1+2\eta-\ep}(\cb)$ for any small $\ep>0$, since
\bean
\| \der (f(Y)\cdot f'(Y))_{us}\|_{\cb} &\leq & c\, \| (\der Y)_{us}\|_\cb \ \leq \ c\, \{ \| (\delha Y)_{us}\|_\cb+\|a_{us} Y_s\|_\cb \}\\
&\leq &c\, \{|u-s|^{2\eta} \cn[Y;\cac^{2\eta}(\cb_\infty)]+|u-s|^\ga \cn[Y;\cac^0(\cb_\ga)]\},
\eean
where we have used (\ref{hold-sob}) to get the last inequality (recall that $a_{ts}:=S_{t-s}-\id$).

\smallskip

\noindent
Then, as far as $L^W N$ is concerned, let us expand $N$ using standard differential calculus, which provides us with the expression
\begin{multline}\label{defi:proc-n}
N_{us}=\int_0^1 dr \, f'(Y_s+r(\der Y)_{us})\cdot  \lcl K^Y_{us}+a_{us}Y_s+L^{aW}_{us}(f(Y_s)) \rcl \\
+\int_0^1 dr \, \lc f'(Y_s+r(\der Y)_{us})-f'(Y_s) \rc \cdot (\der W)_{us} \cdot f(Y_s),
\end{multline}
where the additional operator-valued process $L^{aW}$ is defined by
\begin{equation*}
L^{aW}_{ts}(\vp):= \int_s^t a_{tu}( \vp \cdot dW_u)=a_{ts}(\vp \cdot (\der W)_{ts})-\int_s^t \Delta S_{t-u}(\vp \cdot (\der W)_{tu}) du.
\end{equation*}
Now, since $L^W \in \cac_2^{\frac12-\ep}(\cl(\cb,\cb))$, it is sufficient to prove that $N\in \cac_2^{\frac12+\ep}(\cb)$ for some small $\ep>0$. But, with the expansion (\ref{defi:proc-n}) in hand, this becomes an easy consequence of the a priori controls given by Lemma \ref{lem:apriori}, together with the regularity property:
\begin{equation*}
\cn[L^{aW};\cac_2^{(\frac{1}{2}+\eta)-\ep}(\cl(\cb_{1/2},\cb))]\leq c_{\ep,\tilde{\ep},p}\cn[W;\cac^{\frac{1}{2}-\tilde{\ep}}(\cb_{\eta,2p})],
\end{equation*}
derived from (\ref{prod-sobol-1}). Note in particular how important the assumption $\ga >1/2$, insofar as, by (\ref{hold-sob}),
$$\|f'(Y_s+r(\der Y)_{us})\cdot (a_{us} Y_s)\|_\cb\leq C_f \|a_{us} Y_s\|_{\cb} \leq C |u-s|^{\ga}\, \cn[Y;\cac^0(\cb_\ga)].$$
%Observe for instance that, owing to (\ref{control-l-w}) and (\ref{sobolev-inclusion-2}), one has, for $\ep >0$ small enough,
%\bean
%\lefteqn{\Big\|  \int_0^1 dr \, \lc f'(Y_s+r(\der Y)_{us})-f'(Y_s) \rc \cdot (\der W)_{us} \cdot f(Y_s)  \Big\|_\cb}\\
%&\leq & c_{f} \norm{(\der W)_{us}}_\cb \norm{(\der Y)_{us}}_{\cb_\infty}\\
%&\leq & c_{f,W}  \lln u-s \rrn^{\frac{1}{2}-\ep} \lcl  \norm{(\delha Y)_{us}}_{\cb_\infty} +\norm{a_{us}Y_s}_{\cb_\infty} \rcl\\
%&\leq & c_{f,W} \lcl \lln u-s \rrn^{\frac12+2\eta-\ep} \cn[Y;\cacha^{2\eta}(\cb_\infty)]+\lln u-s \rrn^{\frac12-\ep} \norm{a_{us}Y_s}_{\cb_{\frac{1}{4}+\ep}} \rcl\\
%&\leq & c_{f,W} \lcl \lln u-s \rrn^{1+2\eta-2\ep} \cn[Y;\cacha^{2\eta}(\cb_\infty)]+\lln u-s \rrn^{\frac{1}{4}+\ga-2\ep}  \cn[Y;\cac^0(\cb_\ga)] \rcl.
%\eean

We are thus in a position to apply $\Laha$ to $R^Y$, and so $Z$ is properly defined through (\ref{defin-z}). The regularity of $Z$ and the bound (\ref{bound-z}) are immediate consequences of (\ref{control-l-w})-(\ref{control-l-w-w}) and the contraction property (\ref{contraction-laha}) of $\Laha$. The details are left to the reader.

\end{proof}

\begin{remark}\label{rk:optim}
Although not optimal, the two regularity results (\ref{apriori2-1}) and (\ref{apriori2-2}) are thus sufficient for us to prove that the right-hand side of the decomposition (\ref{representation}) is indeed well-defined. We also retrieve an important stability phenomenon here: $Y$ and $Z$ both belong to the same space $\cacha^{2\eta}(\cb_\infty) \cap \cac^0(\cb_\ga)$. A posteriori, this accounts for our choice in favor of this particular topology.
\end{remark}

We can eventually proceed to prove Theorem \ref{theo:identif}.

\begin{proof}[Proof of Theorem \ref{theo:identif}]
We need to identify the increments of $Y$ with those of the process $Z$ defined in Lemma \ref{cor:def-z}. To do so, we naturally rely on some expansion of the right-hand side of (\ref{eq:336}). Precisely, we have that
\bean
\lefteqn{\int_s^t S_{t-u}(f(Y_u)\cdot dW_u)+\int_s^t S_{t-u}(P \cdot f(Y_u)\cdot f'(Y_u)) \, du}\\
&=&L^W_{ts}(f(Y_s))+L^{WW}_{ts}(f(Y_s)\cdot f'(Y_s))+J^Y_{ts},
\eean
with
\begin{equation}
J^Y_{ts}:=\int_s^t S_{t-u}(P \cdot \der(f(Y)\cdot f'(Y))_{us}) \, du+\int_s^t S_{t-u}(N^Y_{us} \cdot dW_u),
\end{equation}
where the process $N^Y_{ts}=\der(f(Y))_{ts}-(\der W)_{ts} \cdot f(Y_s) \cdot f'(Y_s)$ has already been considered in the proof of Lemma \ref{cor:def-z}. Therefore, with this notation, it holds that
$$\delha(Z-Y)=\Laha_{ts}(R^Y)-J^Y_{ts}.$$
Now, by the contraction property (\ref{contraction-laha}), we know that $\Laha(R^Y) \in \cac_2^{\mu_1}(\cb)$ for some $\mu_1 >1$. Besides, with the same ingredients as in the proof of Lemma \ref{lem:apriori} (Burkholder-Davis-Gundy inequality plus Lemma \ref{lem-grr}, see Appendix B), we can easily lean on the expansion (\ref{defi:proc-n}) of $N$ to  prove that $J^Y\in \cac_2^{\mu_2}(\cb)$ for some $\mu_2>1$
(note that $\delha J^Y=R^Y$). Consequently, $\delha(Z-Y) \in \cac_2^\mu(\cb)$ with $\mu=\inf(\mu_1,\mu_2)>1$, and this entails that $\delha(Z-Y) =0$. Indeed, for any partition $\cp_{[s,t]}=\{s =t_1 < \ldots <t_n=t\}$ of $[s,t]$, one has, due to the telescopic sum property reported in Proposition \ref{prop:telesc},
$$\|\delha(Z-Y)_{ts}\|_{\cb} \leq \sum_{i}\|\delha(Z-Y)_{t_{i+1}t_i}\|_{\cb} \leq c \sum_i |t_{i+1}-t_i|^\mu \leq |\cp_{[s,t]}|^{\mu-1} \lln t-s\rrn,$$
and we conclude by letting the mesh $|\cp_{[s,t]}|:=\max_i |t_{i+1}-t_i|$ tend to $0$.
\end{proof}

As a straightforward consequence of the decomposition (\ref{representation}), we can exhibit an almost sure bound for $Y$ in terms of $W$. Indeed, by plugging the estimate (\ref{bound-z}) back into the equation, we deduce that for any subinterval $I=[\ell_1,\ell_2]\subset [0,T]$,
$$\cn[Y;\cacha^{2\eta}(I;\cb_\infty)]+\cn[Y;\cac^0(I;\cb_\ga)] \leq \|Y_{\ell_1}\|_{\cb_\ga}+C_W |I|^\la \cn[Y;\cq(I)]$$
for some constant $\la >0$, and similar estimates for $K^Y=L^{WW}(f(Y) \cdot f'(Y))+\Laha(R^Y)$ finally show that
$$\cn[Y;\cq(I)] \leq \|Y_{\ell_1}\|_{\cb_\ga}+C_W |I|^\la \cn[Y;\cq(I)].$$
At this point, a basic patching argument easily leads us to the following statement:

\begin{corollary}\label{cor:contr-sol}
Under the hypotheses of Theorem \ref{theo:identif}, there exist $\ep>0$ and $p\geq 1$ such that
\begin{equation}\label{contr-sol}
\cn[Y;\cq([0,T])]
\leq G_{\ep,p}\big( \norm{\psi}_{\cb_\ga},\cn[W;\cac^{\frac{1}{2}-\ep}(\cb_{\eta,2p})] \big)
\end{equation}
for some deterministic function $G_{\ep,p}:(\R^+)^2 \to \R^+$ bounded on bounded sets.
\end{corollary}

%%%%%%%%%%%%%%%%%%%%%%%%%%%%%%%%%%%%%%%%%%%%%%%%%%%%%%%%%%%%%%%%%%%%%%%%%%%%%%%%%%%%%%%%%%%%%%%%%%%%%%%%%%%%%%%%%%%

\

\subsection{Comparison with smooth solutions}
\label{sec:continuity}

The previous considerations will allow us to prove our continuity result (Theorem \ref{thm:continuity}) and for this purpose, we first go back to the case where the driving noise is an absolutely continuous process $\widetilde{W}$ (with values in $\cb_{\eta,2p}$), \emph{assumingly defined on the same probability space as $W$}. In this situation, our mild equation is naturally understood in a pathwise sense as a classical (Riemann-Lebesgue) mild equation, i.e.,
\begin{equation}\label{eq-reg}
\Yti_t=S_t\psiti+\int_0^t S_{t-u}(f(\Yti_u) \cdot d\Wti_u)=S_t \psiti+\int_0^t S_{t-u}(f(Y_u) \cdot \Wti'_u) \, du,
\end{equation}
and the (pathwise) existence and uniqueness of the solution $\Yti$ follows from standard PDE results. The key step towards a comparison between $Y$ and $\Yti$ lies in the following result, which points out the similarity between the couple $(L^W,L^{WW})$ at the core of the previous considerations and the couple $(L^{\Wti},L^{\Wti\Wti})$ constructed from $\Wti$:

\begin{lemma}
Define the operator-valued processes $L^{\Wti}$ and $L^{\Wti\Wti}$ in the classical Riemann-Lebesgue sense as
\begin{equation}\label{defi:l-w-ti}
L^{\widetilde{W}}_{ts}(\vp):=\int_s^t S_{t-u}( \vp \cdot d\widetilde{W}_u)\quad , \quad L^{\widetilde{W}\widetilde{W}}_{ts}(\vp):=\int_s^t S_{t-u}\lp \vp \cdot  (\der \widetilde{W})_{us} \cdot d\widetilde{W}_u \rp,
\end{equation}
for every smooth function $\vp$. Then both formulas (\ref{ipp-l-w}) and (\ref{ipp-l-w-w}) remain valid when substituting $\widetilde{W}$ for $W$, and accordingly the bounds (\ref{control-l-w}) and (\ref{control-l-w-w}) hold true for $\widetilde{W}$ as well.
\end{lemma}

\begin{proof}
It suffices to replace the use of Itô's formula in the proof of Lemma \ref{lem:ipp} with standard integration by parts. Indeed, as an absolutely continuous process, $\Wti$ obeys the rules of standard differential calculus and one has for instance
$$\mathbf{\Wti^{2}}_{ts}:=\int_s^t (\der \Wti)_{us} \cdot d\Wti_u=\frac12 (\der \Wti)_{ts}^2.$$
Consequently, it holds that
\bean
L^{\Wti\Wti}_{ts} \vp &=& \int_s^t S_{t-u}(\vp \cdot d_u(\mathbf{\Wti^{2}}_{us}))\\
&=& \int_s^t S_{t-u}(\vp \cdot d_u(\mathbf{\Wti^{2}}_{us}-\mathbf{\Wti^{2}}_{ts}))\\
&=& \frac12 S_{t-u}(\vp \cdot (\der \Wti)^2_{ts})-\frac12 \int_s^t \Delta S_{t-u}(\vp \cdot [(\der \Wti)^2_{ts}-(\der \Wti)^2_{us}]) du\\
&=& \frac12 S_{t-u}(\vp \cdot (\der \Wti)^2_{ts})-\frac12 \int_s^t \Delta S_{t-u}(\vp \cdot \big[(\der \Wti)^2_{us}+2(\der \Wti)_{tu} \cdot (\der \Wti)_{us}\big]) \,du,
\eean
which precisely fits the pattern of (\ref{ipp-l-w-w}).
\end{proof}

Another consequence of the similarity between $(L^W,L^{WW})$ and $(L^{\Wti},L^{\Wti\Wti})$ through the two formulas (\ref{ipp-l-w}) and (\ref{ipp-l-w-w}) is a set of (readily-checked) Lipschitz-type bounds: with the notations of Proposition \ref{prop:cont-op}, one has, for some polynomial expression $c_{W,\Wti}$,
\begin{equation}\label{contin-bound-1}
\cn[L^W-L^{\Wti};\cac_2^{\frac12-\ep}(\cl(\cb,\cb))]\leq c_{W,\Wti}\, \cn[W-\Wti;\cac^{\frac12-\tilde{\ep}}(\cb_{\eta,2p})],
\end{equation}
\begin{equation}\label{contin-bound-2}
\cn[L^{WW}-L^{\Wti\Wti};\cac_2^{1-\ep}(\cl(\cb,\cb))]\leq c_{W,\Wti}\, \cn[W-\Wti;\cac^{\frac12-\tilde{\ep}}(\cb_{\eta,2p})],
\end{equation}
and this bound remains valid for all of the other topologies involved in Proposition \ref{prop:cont-op}.

\smallskip

\noindent
Then, as far as the solution $\Yti$ is concerned, note that
$$(\delha \Yti)_{ts}=L^{\Wti}_{ts}f(\Yti_s)+L^{\Wti \Wti}_{ts}(f(\Yti_s) \cdot f'(Y_s))+J^{\Yti}_{ts}$$
with $J^{\Yti}_{ts}:=\int_s^t S_{t-u}\big(\big[\der f(\Yti)_{us}-(\der \Wti)_{us} \cdot f(Y_s)\cdot f'(Y_s)\big] \cdot d\Wti_u\big)$, and it is obvious in this (absolutely continuous) situation that $J^{\Yti} \in \cac_2^\mu(\cb)$ for some $\mu >1$. Therefore, we can easily follow the lines of our previous identification procedure (see the proofs of Lemma \ref{cor:def-z} and Theorem \ref{theo:identif}) in order to exhibit a similar formula for the $\delha$-variations of $\Yti$:

\begin{lemma}
Under the hypotheses of Theorem \ref{thm:continuity}, assume that $\widetilde{\psi} \in \cb_\ga$. Then the $\delha$-variations of the solution $\widetilde{Y}$ to (\ref{eq-reg}) can be expanded as
\begin{equation}\label{representation-ti}
(\delha \widetilde{Y})_{ts}=L^{\widetilde{W}}_{ts}(f(\widetilde{Y}_s))+L^{\widetilde{W}\widetilde{W}}_{ts}(f(\widetilde{Y}_s)\cdot f'(\widetilde{Y}_s))+\Laha_{ts}\big(R^{\widetilde{Y}} \big),
\end{equation}
where $R^{\widetilde{Y}}_{tus}:=-\delha \big( L^{\widetilde{W}}f(\widetilde{Y})+L^{\widetilde{W}\widetilde{W}}\big(f(\widetilde{Y})\cdot f'(\widetilde{Y})\big)\big)_{tus}$. In particular, the bound (\ref{contr-sol}) remains valid for $\Yti$ when replacing $\psi$ (resp. $W$) with $\psiti$ (resp. $\Wti$).
\end{lemma}

\

With these identifications in hand, the proof of Theorem \ref{thm:continuity} becomes a matter of a standard rough-paths argument, and we only sketch out the main steps of the procedure (see e.g. the proof of \cite[Lemma 5.2]{RHE-glo} for further details on the computations).

\

\begin{proof}[Proof of Theorem \ref{thm:continuity}]

In order to compare $Y$ with $\Yti$, we can now rely on their respective decompositions (\ref{representation}) and (\ref{representation-ti}). By setting $g:=ff'$, we get that
\begin{multline}\label{decomp-contin}
\delha(Y-\Yti)_{ts}=\big\{ \big[L^W_{ts}-L^{\Wti}_{ts}\big] f(Y_s) +\big[L^{WW}_{ts}-L^{\Wti\Wti}_{ts}\big](g(Y_s)\big\}\\
+\big\{L^{\Wti}_{ts}\big[f(Y_s)-f(\Yti_s)\big]+L^{\Wti\Wti}_{ts} \big[g(Y_s)-g(\Yti_s)\big] \big\}+\Laha_{ts}\big( R^Y-R^{\Yti} \big),
\end{multline}
with a similar splitting for $R^Y-R^{\Yti}$ (based on the expansion (\ref{defi:proc-n})). Now, as in Lemma \ref{cor:def-z}, we consider the following appropriate topology:
$$
\cn[Y-\tilde{Y};\cq(I)]:=\cn[Y-\tilde{Y};\cacha^{2\eta}(I;\cb_\infty)]+\cn[Y-\tilde{Y};\cac^{0}(I;\cb_\ga)]
+\cn[K^Y-K^{\tilde{Y}};\cac_2^{\frac{1}{2}+\eta}(I;\cb)].
$$
By using the decomposition (\ref{decomp-contin}) and the bounds (\ref{contin-bound-1})-(\ref{contin-bound-2}), standard differential calculus shows that for any subinterval $I=[\ell_1,\ell_2]$ of $[0,T]$,
\begin{multline*}
\cn[Y-\tilde{Y};\cq(I)]\\
 \leq C_{W,\tilde{W},\psi,\tilde{\psi}} \lcl \norm{Y_{\ell_1}-\tilde{Y}_{\ell_1}}_{\cb_\ga}+\cn[W-\tilde{W};\cac^{\frac{1}{2}-\ep}(\cb_{\eta,2p})]+
\lln I\rrn^\la \cn[Y-\tilde{Y};\cq(I)]\rcl,
 \end{multline*}
for some constant $\la>0$. As in Corollary \ref{cor:contr-sol}, we can then rely on an elementary patching argument to reach the global bound (\ref{cont-ito}).

\end{proof}

\begin{remark}
The above strategy sheds new light on the classical Itô-Stratonovich correction phenomenon arising in the approximation of stochastic heat equations. Indeed, on the one hand, it emphasizes that the convergence of $\widetilde{Y}$ towards $Y$ reduces to the convergence of $(L^{\widetilde{W}},L^{\widetilde{W}\widetilde{W}})$ towards $(L^W,L^{WW})$, and on the other, continuous bounds such as (\ref{contin-bound-2}) clearly highlight the relevance of the Stratonovich interpretation of $L^{WW}$ in this context. In a way, the correction phenomenon is therefore more directly observed through the decomposition (\ref{defi:l-w-w}) of $L^{WW}$ as the sum of an Itô integral and a trace term.
\end{remark}

%%%%%%%%%%%%%%%%%%%%%%%%%%%%%%%%%%%%%%%%%%%%%%%%%%%%%%%%%%%%%%%%%%%%%%%%%%%%%%
%%%%%%%%%%%%%%%%%%%%%%%%%%%%%%%%%%%%%%%%%%%%%%%%%%%%%%%%%%%%%%%%%%%%%%%%%%%%%%%%%%%%%%%%%%%%%%

\section{Approximations in law}
\label{sec:approx-law}

We now aim to prove our approximation result, that is Theorem \ref{theo:approx}. Thus, from now on, we assume that the hypotheses in Theorem \ref{theo:approx} 
are all satisfied. Recall that the approximation processes involved in this statement, namely the Donsker and the Kac-Stroock approximations, 
have been specified in the Introduction (see (\ref{dons-intro}) and (\ref{ks-intro})), as well as the notations $\mathbf{W}$ and $\be^{n,\cdot}$. 
Besides, in this part of the paper we take $T=1$ for the sake of simplicity.

%%%%%%%%%%%%%%%%%%%%%%%%%%%%%%%%%%%%%%%%%%%%%%%%%%%%%%%%%%%%%%%%%%%%%%%%%%

\subsection{Preliminary results}
\label{sec:prel-res}

As a first step towards Theorem \ref{theo:approx}, we need to check that the processes we have constructed via $\mathbf{W}$ are indeed well-defined. 
To do so, we will make use of the following bound. 

\begin{lemma}\label{lem:hypercontra}
Fix $n\geq 1$. Let $X_1^{(n)},\ldots,X_n^{(n)}$
 be independent centered random variables with moments of any order and $f_n:\{1,\ldots,n\} \to \R$.
 Then for every  $r\ge 1$, there exists a constant $C_r$ which only depends on $r$
 such that
$$E\bigg[ \big| \sum_{i=1}^n f_n(i) X_i^{(n)} \big|^{2r} \bigg] \leq C_r   \bigg( \sum_{i=1}^n f_n(i)^2\bigg)^{r} \cdot \bigg( \sup_{1\leq i\leq n} E\big[ |X_i^{(n)}|^{2r} \big] \bigg).$$
\end{lemma}

This inequality can be easily deduced from the following result, which is
clear for $r=1$  and was proved by Rosenthal for $r>1$ (see \cite[Thm. 3]{Rosenthal}).

\begin{theorem}\label{prop:Rosenthal}
Let $Y_1,\ldots, Y_n$ be independent centered random variables
satisfying 
 $E\big[|Y_i|^{2r}\big]<\infty$, where $r\ge 1$. Then, there exists a constant
$C_r$ such that
$$ E\bigg[ \big|\sum_{i=1}^n Y_i \big|^{2r}\bigg] \le
C_r\,\max\Big\{\sum_{i=1}^n E|Y_i|^{2r},\,\bigg(\sum_{i=1}^n
E|Y_i|^{2}\bigg)^r\Big\}.$$
\end{theorem}

\

The transition from real-valued to $\cb_{\eta,2p}$-valued processes will be ensured by the following result.

\begin{proposition}\label{prop:conver}
Let $(X_k)_{k\geq 1}$ be a sequence of centered i.i.d. random variables
on some probability space $(\Omega,\mathcal{F},P)$. Assume that each $X_i$
 has moments of any order, and consider a sequence $(\la_k)_{k\geq 1}$
  of positive numbers such that $\sum_{k\geq 1} \la_k\,  k^{4\eta} < \infty$ for some (fixed) $\eta>0$. Then, for every $p,q\geq 1$, the random series of functions $\sum_k \sqrt{\la_k} X_k e_k$ converges in $L^{2pq}(\Omega,\cb_{\eta,2p})$ to an element $\mathbf{X}$ which satisfies
\begin{equation}\label{bound-x-gras}
E\big[ \|\mathbf{X}\|_{\cb_{\eta,2p}}^{2pq} \big]\leq C_{p,q,\la,\eta} E\big[ |X_1|^{2pq}\big],
\end{equation}
for some constant $C_{p,q,\la,\eta}$ which only depends on $p$, $q$
and $\sum_{k\geq 1} \la_k \,  k^{4\eta}$.
\end{proposition}

\begin{proof}
Set $\mathbf{X}^n:=\sum_{k=1}^n \sqrt{\la_k} X_k e_k$ and observe
that $\| \mathbf{X}^m-\mathbf{X}^n
\|_{\cb_{\eta,2p}}=\|\mathbf{X}^{(m,n),\eta} \|_{L^{2p}(0,1)}$, where we
have set $\mathbf{X}^{(m,n),\eta}(\xi):=\sum_{k=n+1}^m k^{2\eta}
\sqrt{\la_k} X_k e_k(\xi)$. Then, by Jensen's inequality,
$$
E\big[\| \mathbf{X}^{(m,n),\eta} \|_{L^{2p}(0,1)}^{2pq}\big]=
 E\bigg[ \bigg( \int_0^1 d\xi \, |\mathbf{X}^{(m,n),\eta}(\xi)|^{2p}\bigg)^q \bigg]
  \leq  \int_0^1 d\xi \, E\big[|\mathbf{X}^{(m,n),\eta}(\xi)|^{2pq}\big],
$$
and thanks to Lemma \ref{lem:hypercontra}, we get
\begin{eqnarray}
E\big[\| \mathbf{X}^{(m,n),\eta} \|_{L^{2p}(0,1)}^{2pq}\big]&\leq & C_{p,q} E\big[|X_1|^{2pq}\big]
\int_0^1 d\xi \, \bigg( \sum_{k=n+1}^m \la_k \, k^{4\eta}\, e_k(\xi)^2 \bigg)^{pq} \nonumber\\
&\leq & C_{p,q} E\big[|X_1|^{2pq}\big] \bigg( \sum_{k=n+1}^m \la_k
\, k^{4\eta} \bigg)^{pq}\label{bound-interm}
\end{eqnarray}
due to the uniform bound $\| e_k\|_{\cb_\infty} \leq \sqrt{2}$. In particular, $E\big[\| \mathbf{X}^m-\mathbf{X}^n \|_{\cb_{\eta,2p}}^{2pq}\big]$ tends to 
zero as both $m$ and $n$ tend to infinity, so that $\mathbf{X}^n$ converges in $L^{2pq}(\Omega,\cb_{\eta,2p})$. 
The bound (\ref{bound-x-gras}) can of course be derived from (\ref{bound-interm}).
\end{proof}

In particular, due to Hypothesis \ref{hypo-noise-3}, we can conclude that $W^n=\mathbf{W}(S^{n,\cdot})$ and $W^n=\mathbf{W}(\theta^{n,\cdot})$ are indeed well-defined processes with values in $\cb_{\eta,2p}$. Let us now get a little bit closer to the assumptions of Theorem \ref{thm:continuity} 
by checking that in both cases, $W^n$ admits an absolutely continuous version.

\begin{lemma}\label{acDonsker}
For any fixed $n\geq 1$, both the Donsker approximation $W^n=\mathbf{W}(S^{n,\cdot})$ and the Kac-Stroock approximation $W^n=\mathbf{W}(\theta^{n,\cdot})$
have an absolutely continuous version with values in $\bepp$, for all $p\geq 1$.
\end{lemma}

\begin{proof}
Since the (deterministic) approximation grid for $S^{n,k}$ does not depend on $k$, it is easily seen that
$$\mathbf{W}(S^{n,\cdot})_t=\mathbf{W}(S^{n,\cdot})_{\frac{i}{n}}+n\cdot \Big(t-\frac{i}{n}\Big)\cdot \big\{\mathbf{W}(S^{n,\cdot})_{\frac{i+1}{n}}-\mathbf{W}(S^{n,\cdot})_{\frac{i}{n}}\big\} \qquad \text{if} \ t\in \Big[\frac{i}{n},\frac{i+1}{n}\Big].$$
In particular, $\mathbf{W}(S^{n,\cdot})$ is a piecewise linear process (with values in $\cb_{\eta,2p}$) and accordingly it is absolutely continuous.

As far as the Kac-Stroock approximation is concerned, first we can
see that it has a continuous version with values in $\bepp$. Indeed,
applying Proposition \ref{prop:conver},
\begin{align*}
E\lc\normg{\delta(\mathbf{W}(\theta^{n,\cdot}))_{ts}}_{\bepp}^{2pq}\rc&\le
C\,
E\Big[\big|\delta(\theta^n)_{ts}\big|^{2pq}\Big]\\
&= C\, E\Big[\Big|\int_s^t
\sqrt{n}\,(-1)^{\zeta+N(nu)}du\Big|^{2pq}\Big]\le C n^{pq}|t-s|^{2pq}.
\end{align*}
For the sake of clarity, we
will also denote by $\mathbf{W}(\theta^{n,\cdot})$ this continuous
version.  To prove the existence of an absolutely continuous
version, we will see that  with probability 1,
\begin{equation}\label{repr-kac}
\mathbf{W}(\theta^{n,\cdot})_t=\int_0^t \mathbf{W}(\dot{\theta}^{n,\cdot})_s
\, ds,\quad \text{ for any } t\in [0,1],
\end{equation}
where $\dot{\theta}^n_t:=\sqrt{n} \cdot (-1)^{\zeta+N(nt)}$. Indeed, thanks to Proposition \ref{prop:conver},
 $\mathbf{W}(\dot{\theta}^{n,\cdot})_t$ is well-defined for every $t\in [0,1]$ as
 an element of $L^{2p}(\Omega,\cb_{\eta,2p})$ and
$$ E\bigg[  \int_0^1\!
\|\mathbf{W}(\dot{\theta}^{n,\cdot})_s \|_{\cb_{\eta,2p}} ds \bigg]\le
\int_0^1\!\!\! \lp E\big[\|\mathbf{W}(\dot{\theta}^{n,\cdot})_s
\|_{\cb_{\eta,2p}}^{2p}\big]\rp^{\frac1{2p}} ds\ \leq \ C_p \int_0^1
\!\!\!\lp E\big[|\dot{\theta}^{n}_s |^{2p} \big] \rp^{\frac1{2p}}ds\ <
\ \infty.
$$
As a consequence $\mathbf{W}(\dot{\theta}^{n,\cdot})$ is (a.s.)
Bochner-integrable. Moreover,  for each $t\in [0,1]$,
$$\mathbf{W}(\theta^{n,\cdot})_t=\lim_{N\to \infty} \sum_{k=1}^N \sqrt{\la_k}\, \theta^{n,k}_t e_k=
\lim_{N\to \infty} \int_0^t \Big( \sum_{k=1}^N \sqrt{\la_k}\,
\dot{\theta}^{n,k}_s e_k \Big) ds \qquad \text{in} \
L^{2p}(\Omega,\cb_{\eta,2p})$$ and
\begin{multline*}
E\bigg[\Big\|\int_0^t \!\!\!\Big( \sum_{k=1}^N \sqrt{\la_k}\,
\dot{\theta}^{n,k}_s e_k \Big) ds -\int_0^t
\mathbf{W}(\dot{\theta}^{n,\cdot})_s \, ds \Big\|^{2p}_{\cb_{\eta,2p}}
\bigg] \leq \int_0^1 E\Big[ \big\| \sum_{k=N+1}^\infty\!\!\!
\sqrt{\la_k}\, \dot{\theta}^{n,k}_s e_k \big\|^{2p}_{\cb_{\eta,2p}}
\Big]ds \\
\le
C_{pq}\Big(\sup_{s\in[0,1]}E\big[|\dot\theta^n_s|^{2pq}\big]\Big)\Big(\sum_{k=N+1}^{\infty}\la_k\,k^{4\eta}\Big),
\end{multline*}
by  similar arguments as in the proof of Proposition
\ref{prop:conver}. Since the last expression tends to $0$ as $N\to
\infty$, we obtain that for each $t\in [0,1]$
$$
\mathbf{W}(\theta^{n,\cdot})_t=\int_0^t \mathbf{W}(\dot{\theta}^{n,\cdot})_s
\, ds\quad \text {a.s.}$$
Thus,  since 
$\{\mathbf{W}(\theta^{n,\cdot})_t,\,t\in [0,1]\}$ and $\{\int_0^t
\mathbf{W}(\dot{\theta}^{n,\cdot})_s \, ds,\,t\in [0,1]\}$ are both
continuous processes, we can conclude that
$$
P\Big\{\mathbf{W}(\theta^{n,\cdot})_t=\int_0^t
\mathbf{W}(\dot{\theta}^{n,\cdot})_s \, ds,\; \forall
t\in[0,1]\Big\}=1.$$
\end{proof}

%%%%%%%%%%%%%%%%%%%%%%%%%%%%%%%%%%%%%%%%%%%%%%%%%%%%%%%%%%%%%%%%%%%%%%%%%%%%%%%%%%%%%

\subsection{A general convergence criterion}
\label{sec:criterion}

One of our key ingredients to prove Theorem \ref{theo:approx} via Theorem \ref{thm:continuity} lies in the following statement, which puts forward sufficient conditions for an approximation of the noise (defined on the same probability space) to converge with respect to the topology involved in (\ref{cont-ito}).

\begin{proposition}\label{prop:crit}
Let $(\be^n)_{n\geq 1}$ be a sequence of centered processes and $\be$ a Brownian motion,
 all defined on a same probability space $(\Omega,\mathcal{F},P)$, and such that the following two
  conditions are satisfied:
\begin{itemize}
\item[(i)] For every integer $p\geq 1$, there exists a constant
$C_p$ such that for all $s,t\in [0,1]$ and all $n\geq 1$,
$$E\big[|\be^n_t-\be^n_s|^{2p}\big] \leq C_p \lln t-s\rrn^{p}.$$

\item[(ii)] For every integer $p\geq 1$, there exists a constant
$C_p$ such that for all $n\geq 1$,
$$\sup_{t\in [0,1]} E\big[ |\be^n_t-\be_t |^{2p}\big] \leq C_p n^{-\nu p},$$
for some fixed parameter $\nu >0$.
\end{itemize}
Then if we consider independent copies $(\be^{n,k})_{k\geq 1}$ (resp. $(\be^k)_{k\geq 1}$) of $\be^n$ (resp. $\be$) on a same probability space, we have that, for any integer $p\geq 1$ and any $\ep>0$,
$$\cn[\mathbf{W}(\be^{n,\cdot})-\mathbf{W}(\be^\cdot);\cac^{\frac12-\ep}(\cb_{\eta,2p})] \underset{n\to \infty}{\longrightarrow }0 \qquad \text{a.s.}$$
\end{proposition}

\smallskip

Let us first see how to combine the above conditions (i) and (ii) so as to exhibit convergent bounds in Hölder topology.

\begin{lemma}\label{lem:estim}
Under the hypotheses of Proposition \ref{prop:crit}, for all
integers $n,p\geq 1$, all $\ep\in (0,1)$ and $s<t\in [0,1]$, one has
$$E\big[ |\der(\be^n-\be)_{ts}|^{2p} \big] \leq C_p \frac{|t-s|^{(1-\ep)p}}{n^{p\ep\nu}},$$
for some constant $C_p$ which only depends on $p$.
\end{lemma}

\begin{proof}
If $|t-s|\leq n^{-\nu}$, then due to the condition (i), it holds that
$$
E\big[ |\der(\be^n-\be)_{ts}|^{2p} \big] \leq  C_p \big\{ E\big[ |\be^n_t-\be^n_s|^{2p} \big] +E\big[ \be_t-\be_s|^{2p} \big] \big\} \leq C_p \, |t-s|^p  \leq C_p \frac{|t-s|^{(1-\ep)p}}{n^{\nu \ep p}}.
$$
On the other hand, if $|t-s| >n^{-\nu}$, one has, thanks to the condition (ii),
$$
E\big[ |\der(\be^n-\be)_{ts}|^{2p} \big] \leq C_p \sup_{t\in [0,1]} E\big[ |\be^n_t-\be_t|^{2p}\big] \leq C_p \, n^{-\nu p} \leq C_p \frac{|t-s|^{(1-\ep)p}}{n^{\nu \ep}}.
$$
\end{proof}

\begin{proof}[Proof of Proposition \ref{prop:crit}]
By using successively Proposition \ref{prop:conver} and Lemma \ref{lem:estim}, we get, for any $q\geq 1$,
$$E\Big[ \| \der\big( \mathbf{W}(\be^{n,\cdot})-\mathbf{W}(\be^.) \big)_{ts}\|_{\cb_{\eta,2p}}^{2pq} \Big]\leq C_{p,q,\eta} \, E\Big[ |\der(\be^n-\be)_{ts}|^{2pq} \Big] \leq C_{p,q,\eta} \frac{|t-s|^{(1-\ep)pq}}{n^{pq\nu \ep}}.$$
We are thus in a position to apply the Garsia-Rodemich-Rumsey Lemma \ref{lem-grr} (with $\der^\ast=\der$) and assert that, for $q$ large enough,
\bean
\lefteqn{E\Big[ \cn[\mathbf{W}(\be^{n,\cdot})-\mathbf{W}(\be^\cdot);\cac^{\frac{1}{2}-\ep}(\cb_{\eta,2p})]^{2pq}\Big]}\\
 &\leq & C_{p,q,\eta} \iint_{[0,T]^2}  \frac{E\Big[ \| \der\big( \mathbf{W}(\be^{n,\cdot})-\mathbf{W}(\be^\cdot) \big)_{ts}\|_{\cb_{\eta,2p}}^{2pq} \Big]}{\lln t-s \rrn^{2pq(\frac{1}{2}-\ep)+2}} \, dsdt \\
&\leq & C_{p,q,\eta}\, n^{-\ep pq\nu} \iint_{[0,T]^2} \lln t-s\rrn^{pq\ep-2} dsdt \ \leq \ C_{p,q,\eta}\, n^{-\ep pq\nu}.
\eean
As a result, it holds that
$$P\big( \cn[\mathbf{W}(\be^{n,\cdot})-\mathbf{W}(\be^\cdot);\cac^{\frac{1}{2}-\ep}(\cb_{\eta,2p})]>n^{-\ep \nu/4} \big)\leq C_{p,q,\eta} \, n^{-\ep pq\nu/2},$$
which, thanks to the Borell-Cantelli Lemma, leads us to the conclusion, that is
$$\cn[\mathbf{W}(\be^{n,\cdot})-\mathbf{W}(\be^\cdot);\cac^{\frac{1}{2}-\ep}(\cb_{\eta,2p})] \to 0 \qquad \text{a.s.}$$ as $n$ tends to infinity.
\end{proof}

\medskip

\noindent {\bf{Example:}} 
As an immediate illustration of Proposition
\ref{prop:crit}, let us consider here the \emph{Wong-Zakai}
approximation of a given noise $W$ satisfying Hypothesis
\ref{hypo-noise-3}. Precisely, set
$$W^n_t:=W_{\frac{i}{n}}+n\cdot \Big( t-\frac{i}{n} \Big) \cdot \big\{W_{\frac{i+1}{n}}-W_{\frac{i}{n}} \big\} \quad \text{for} \ t\in \Big[\frac{i}{n},\frac{i+1}{n}\Big],$$
and denote by $Y^n$ the solution of the equation
$$Y^n_t=S_t\psi+\int_0^t S_{t-u}(f(Y^n_u) \cdot dW^n_u),$$
understood in the classical Riemann-Lebesgue sense. Note that $W^n$ can be equivalently described as follows: with the expansion (\ref{repr-sum}) of $W$ in mind, 
i.e. $W=\mathbf{W}(\be^\cdot)$, we have that $W^n=\mathbf{W}(\be^{n,\cdot})$, where, for each $k\geq 1$, $\be^{n,k}$ stands for the linear interpolation of $\be^k$ 
with mesh $\frac{1}{n}$. Therefore, it suffices to check that the conditions (i) and (ii) in Proposition \ref{prop:crit} are satisfied by $\be^n:=\be^{n,1}$, 
which is a matter of elementary computations (it can be also seen as a particular case of the forthcoming Proposition \ref{tightDonsker}). 

Together with Theorem \ref{thm:continuity}, we retrieve the following almost sure approximation result:
\begin{proposition}
Under the hypotheses of Theorem \ref{thm:continuity}, let $Y^n$ be
the Wong-Zakai approximation of (\ref{equa-mild}) with mesh
$\frac{1}{n}$ and initial condition $\psi$. Then, as $n\to \infty$,
one has $\cn[Y-Y^n;\cac^0(\cb_\ga)] \to 0$ a.s.
\end{proposition}
This almost sure result in a non-linear situation is closely related to those of \cite{Brze-Flandoli} or \cite{Bally-Millet-Sanz}, 
where Wong-Zakaï approximations for some parabolic type equations have been considered. We also note that convergence in law for this type of 
approximations in the framework of stochastic evolution equations has been studied in \cite{Tessitore-Zabczyk-JEE2006}.

%In fact, the former paper studies the almost sure 
%convergence of Wong-Zakai type approximations for stochastic evolution equation in $\mathbb{R}^d$ driven by a finite-dimensional Wiener process, 
%while the latter deals with convergence in law for the Wong-Zakai approximations for the same type of equations and covering the case of bounded 
%domains with smooth boundary.  

\

\

Now, let us turn to the proof of the weak approximation results of Theorem \ref{theo:approx}, and which successively involve the Donsker approximation $\be^n=S^n$ and the Kac-Stroock approximation $\be^n=\theta^n$. In both cases, we wish to exploit the criterion of Proposition \ref{prop:crit}, which naturally leads us to the following 2-step procedure:

\medskip

\noindent \textbf{Step 1}: Show that Condition (i) is satisfied, i.e., $\sup_n E\big[|\be^n_t-\be^n_s|^{2p}\big] \leq C_p \lln t-s\rrn^{p}$.

\smallskip

\noindent \textbf{Step 2}: Find a probability space $(\bar{\Omega},\bar{\mathcal{F}},\bar{P})$, a sequence $\bar{\be}^n$ and a 
Brownian motion $\bar{\be}$, both defined on $(\bar{\Omega},\bar{\mathcal{F}},\bar{P})$, 
such that $\bbe^n\sim \be^n$ and $\sup_{t\in [0,1]} \bar{E}\big[ |\bbe^n_t-\bbe_t |^{2p}\big] \leq C_p n^{-\nu p}$ for some fixed parameter $\nu >0$.

\medskip

Once these two conditions have been checked, the proof of the weak
convergence $Y^n \to Y$ in $\cac^0(\cb_\ga)$ becomes a
straightforward consequence of Theorem \ref{thm:continuity} and
Proposition \ref{prop:crit}, since $\mathbf{W}(\bbe^{n,\cdot}) \sim
\mathbf{W}(\be^{n,\cdot})$ and accordingly, if $\bar{Y}^n$ denotes
the solution of (\ref{regu-equa}) associated with
 $\bar{W}^n:=\mathbf{W}(\bbe^{n,\cdot})$, it holds that $\bar{Y}^n \sim Y^n$.

Note that for both approximations $S^n$ and $\theta^n$, the result
 in \textbf{Step 2} will be derived from a Skorokhod embedding argument
  (see \cite{Skorokhod}). In the Donsker situation (Proposition \ref{mateixespai-Donsker}),
  this relies on a classical strategy towards the celebrated invariance principles
  (see \cite[Section 5.3]{morters-peres}).
 In the Kac-Stroock situation (Proposition \ref{mateixespai-Kac}), we will take advantage of an identification result due to Griego, 
Heath and Ruiz-Moncayo (see \cite{Griego-Heath-Ruiz}).

%%%%%%%%%%%%%%%%%%%%%%%%%%%%%%%%%%%%%%%%%%%%%%%%%%%%%%%%%%%%

\subsection{Donsker approximation }
\label{Donsker}

Here, we proceed to tackle the above 2-step procedure for the Donsker approximation $S^n$. 

\

\noindent
\textbf{Step 1 (Donsker case):}

\begin{proposition}\label{tightDonsker}
For every $p\geq 1$, there exists a positive constant $C_{p}$ such that, for all $0\leq s<t\leq 1$,
\begin{equation}\label{desigmom}
\sup_{n\in \N} E\big[|S^{n}_t-S^{n}_s  |^{2p}\big]\le
C_{p}\abs{t-s}^{p}.
\end{equation}
\end{proposition}

\begin{proof}
First, note that $S^{n}$ can also be expressed as
$$S_t^{n}=n^{1/2}
\sum_{i=1}^n \Big( \int_0^t \mathbf{1}_{[\frac{i-1}{n},\frac{i}{n}]}
(u) \, du\Big) Z_i.$$ Then, by Lemma \ref{lem:hypercontra}, we have
\begin{multline*}
%\label{proo-dons}
E\big[|S^{n}_t-S^{n}_s  |^{2p}\big] = n^p E\bigg[ \bigg|
\sum_{i=1}^n
\Big( \int_s^t \mathbf{1}_{[\frac{i-1}{n},\frac{i}{n}]} (u) \, du\Big) Z_i \bigg|^{2p} \bigg]\\
\leq C_p\, n^p E\big[ |Z_1|^{2p}\big] \bigg( \sum_{i=1}^n \Big(
\int_s^t \mathbf{1}_{[\frac{i-1}{n},\frac{i}{n}]}(u) \, du \Big)^2
\bigg)^p\\
\le C_p \, n^p\Big(\max_{i=1,\ldots,n}\Big\{\int_s^t
\mathbf{1}_{[\frac{i-1}{n},\frac{i}{n}]}(u) \, du
\Big\}\Big)^p\Big(\sum_{i=1}^n\int_s^t
\mathbf{1}_{[\frac{i-1}{n},\frac{i}{n}]}(u) \, du \Big)^p\le C_p
|t-s|^p.
\end{multline*}
%If $\frac{k}{n} \leq s <t \leq \frac{k+1}{n}$ for some $k\in \{0,\ldots,n-1\}$, then
%$$\sum_{i=1}^n \Big( \int_s^t \mathbf{1}_{[\frac{i}{n},\frac{i+1}{n}]}(u) \, du \Big)^2=|t-s|^2 \leq |t-s| \cdot n^{-1},$$
%while if $\frac{k}{n}\leq s < \frac{k+1}{n} < \ldots < \frac{l}{n} \leq t < \frac{l+1}{n}$,
%$$\sum_{i=1}^n \Big( \int_s^t \mathbf{1}_{[\frac{i}{n},\frac{i+1}{n}]}(u) \, du \Big)^2=\Big( \frac{k+1}{n}-s \Big)^2+\Big( \frac{l-k-1}{n}\Big)\cdot  n^{-1}+\Big( t-\frac{l}{n} \Big)^2 \leq 3 |t-s|  n^{-1}.$$
%Going back to (\ref{proo-dons}), we deduce the uniform bound (\ref{desigmom}).
\end{proof}

\noindent
\textbf{Step 2 (Donsker case):}
\begin{proposition}\label{mateixespai-Donsker}
Let $(Z_i)_{i\in \N}$ be a sequence of i.i.d.
centered random variables with unit variance. Then, there exists a probability space $(\bar
\oom,\bar{\mathcal F},\bar P)$,  a Brownian motion
$\bbe$ defined on it and, for each $n\geq 1$,  a family of independent random variables
$(\bar Z^{(n)}_i)_{i=1,\ldots, n}$ with the same law as
$Z_i$, such that the following is satisfied. Set
$$\bar S_t^n:=n^{-1/2}\lbcl \sum_{j=1}^{i-1} \bar Z_j^{(n)}+\frac{t-(i-1)/n}{1/n}\, \bar
Z_i^{(n)}\rbcl\quad\text{ if }\,
t\in\lbc\frac{i-1}{n},\frac{i}{n}\rbc,\quad \text{with
}\,i\in\{1,\ldots,n\}.$$ Then, for every integer $p\geq 1$,
$$\sup_{t\in[0,1]}E\big[\abs{\bbe(t)-\bar S_t^n}^{2p}\big]\le C_{p}
n^{-p/4}.$$
\end{proposition}

\begin{proof}
As mentioned earlier, it is based on a general Skorokhod embedding theorem (see \cite[p. 163]{Skorokhod}), which, in our particular situation, can be stated as follows : there exists a probability space
$(\bar \oom,\bar{\mathcal F},\bar P)$, a Brownian motion $\bbe$ defined on it and, for each $n\in\N$, a sequence
$\{\tau_i^{(n)}\}_{i= 1,\dots,n}$ of independent and positive random
variables such that the random vector
$$\lbp\bbe\big(\tau_1^{(n)}\big),\bbe\big(\tau_1^{(n)}+\tau_2^{(n)}\big),\ldots,\bbe\big(\taun_1+\cdots+\taun_n\big)\rbp$$
has the same law as
$$\lbp
\frac{1}{\sqrt{n}}Z_1,\frac{1}{\sqrt{n}}\big(Z_1+Z_2\big),\ldots,\frac{1}{\sqrt{n}}\big(Z_1+\cdots+
Z_n\big)\rbp.$$
Moreover, it holds that $E\big[\taun_i\big]=E\big[(
Z_1/{\sqrt{n}})^2\big]=\frac{1}{n}$ and
$$E\big[ |\taun_i|^m\big]\le
C_mE\Big[\Big(\frac1{\sqrt{n}}Z_i\Big)^{2m}\Big]\le
\frac{C_m}{n^m},\quad\text{for any }\,m\in\N.$$
Set $T_0^n:=0$ and $T_i^{(n)}:=\sum_{j=1}^i\taun_j$ for $i\ge 1$. With this notation, we can infer that
$$\bbe\big(\ten_i\big)-\bbe\big(\ten_{i-1}\big)\sim\frac1{\sqrt{n}}
Z_i.$$
We define now
$$\bar
Z_i\pn=\sqrt{n}\big\{ \bbe\big(\ten_i\big)-\bbe\big(\ten_{i-1}\big)\big\} \sim
Z_i$$
and
$$\bar S_t^n=n^{-1/2}\lbcl \sum_{j=1}^{i-1} \bar Z_j \pn+\frac{t-(i-1)/n}{1/n}\, \bar
Z_i\pn\rbcl\quad\text{ if }\,
t\in\lbc\frac{i-1}{n},\frac{i}{n}\rbc.$$
Observe that, if $t\in\lbc\frac{i-1}{n},\frac{i}{n}\rbc$,
$$\bar S_t^n=\bbe\big(\ten_{i-1}\big)+\frac{t-(i-1)/n}{1/n}\, \Big\{ \bbe\big(\ten_i\big)-\bbe\big(\ten_{i-1}\big)\Big\},$$
and hence, for $ t\in\lbc\frac{i-1}{n},\frac{i}{n}\rbc$, we have that
\bean
E\Big[\absg{\bbe(t)-\bar S_t^n}^{2p}\Big]&\le &
C_{p}E\Big[\absg{\bbe(t)-\bbe\big(\ten_{i-1}\big)}^{2p}\Big]+C_{p}E\Big[ \absg{\frac1{\sqrt{n}}Z_i}^{2p}\Big]\\
& \leq&
C_{p}E\Big[\absg{\bbe(t)-\bbe\big(\ten_{i-1}\big)}^{2p}\Big]+C_{p}n^{-p}.
\eean
Thus, we only need to bound the first term in the latter expression, and to this end, we will use the following decomposition:
$$E\Big[\absg{\bbe(t)-\bbe\big(\ten_{i-1}\big)}^{2p}\Big]=A_1^n+A_2^n,$$
with
$$A_1^n=E\Big[\absg{\bbe(t)-\bbe\big(\ten_{i-1}\big)}^{2p}\mathbf{1}_{\{|t-\ten_{i-1}|\le n^{-1/4}\}}\Big]
$$
and
$$A_2^n=E\Big[\absg{\bbe(t)-\bbe\big(\ten_{i-1}\big)}^{2p}\mathbf{1}_{\{|t-\ten_{i-1}|> n^{-1/4}\}}\Big].
$$
On the one hand, the maximal inequality for Brownian motion yields
\begin{eqnarray}
A_1^n\le E\Big[\max_{s\in [(t-n^{-1/4})\vee
0,\,t]}\abs{\bbe(s)-\bbe(t)}^{2p}\Big]+E\Big[\max_{s\in
[t,\,(t+n^{-1/4})\wedge 1]}\abs{\bbe(s)-\bbe(t)}^{2p}\Big]\nonumber\\
\le 2\,E\Big[\max_{h\in [0,n^{-1/4}]}\abs{\bbe(h)}^{2p}\Big]\le
C_{p} E\big[\abs{\bbe(n^{-1/4})}^{2p}\big]\le C_{p} n^{-p/4}.\label{fix-4}
\end{eqnarray}
On the other hand, by Cauchy-Schwarz inequality, we have
\beq\label{A2}
A_2^n\le C_{p}\Big\{
E\big[\bbe\big(\ten_{i-1}\big)^{4p}\big]+E\big[\bbe\big(t\big)^{4p}\big]\Big\}^{1/2}\Big\{
P\big(
\absg{t-\ten_{i-1}}>n^{-1/4}\big)\Big\}^{1/2}
\eeq
Note that, by Lemma \ref{lem:hypercontra},
$$
E\big[\bbe\big(\ten_{i-1}\big)^{4p}\big]=E\Big[
\Big(\frac1{\sqrt{n}}\big(Z_1+\ldots+Z_i\big)\Big)^{4p}\Big]
\le C_p n^{-2p} i^{2p}\le
C_{p}.
$$
Thus, in order to estimate the term $A_2^n$, we only need to study the
probability appearing in (\ref{A2}). To do so, observe first that since
$t\in [\frac{i-1}{n},\frac{i}{n}]$, we have, for $n$ such that
$\frac12\,n^{-1/4}>\frac1{n}$ (that is, for $n\ge 3$),
\begin{equation}\label{ref-donsk-1}
 P\big(
\absg{t-\ten_{i-1}}>n^{-1/4}\big)
\leq
P\Big(
\absg{\ten_{i-1}-\frac{i-1}{n}}>\frac{1}{2}n^{-1/4}\Big).
\end{equation}
Then, using again Lemma \ref{lem:hypercontra}, we get
\begin{eqnarray}
 P\Big(
\absg{\ten_{i-1}-\frac{i-1}{n}}>\frac{1}{2}n^{-1/4}\Big)&\le &
C_{p}\,n^{p/2} E\bigg[\lbp\sum_{j=1}^{i-1}\Big\{\taun_j-\frac{1}{n}\Big\}\rbp^{2p}\bigg]\nonumber\\
&\le & C_{p}\, n^{p/2}i^{p}n^{-2p}\ \le \ C_{p} n^{-p/2}.\label{ref-donsk-2}
\end{eqnarray}
Therefore, $A_2^n\le C_{p}\, n^{-p/4}$, which concludes the proof.
\end{proof}

%%%%%%%%%%%%%%%%%%%%%%%%%%%%%%%%%%%%%%%%%%%%%%%%%%%%%%%%%%%%%%%%%%%%%%%%%%%%%%%%%%%%

\subsection{Kac-Stroock approximation}
\label{Kac}

Along the same lines as in the Donsker case, we proceed now to analyze the Kac-Stroock approximations based on $\theta^n$.

\

\noindent
\textbf{Step 1 (Kac-Stroock case):}
\begin{proposition}\label{tightKac}
For every integer $p\geq 1$, there exists a positive constant
$C_{p}$ such that, for all $0\leq s<t\leq 1$,
\begin{equation}\label{desigmomKac}
\sup_{n\in \N}E\big[  |\theta^{n}_t-\theta^n_s |^{2p}\big]\le
C_{p}\abs{t-s}^{p}.
\end{equation}
\end{proposition}

\begin{proof}
We have that
 \begin{align*}
& E\big[  |\theta^{n}_t-\theta^n_s |^{2p}\big]=E\bigg[\lbp
\sqrt{n}\ist (-1)^{\zeta+N(nu)} du \rbp^{2p}\bigg]=
 n^{p} E\bigg[\lbp \ist (-1)^{N(nu)} du\rbp^{2p}\bigg] \\
& \quad =C_{p} n^{p}E\lbc\ist\cdots\ist
(-1)^{N(nu_{_1})+N(nu_{_2})+\cdots+N(nu_{_{2p}})} du_1\cdots
du_{2p}\rbc \\
& \quad  =C_{p} n^{p}E\lbc\ist\cdots\ist\mathbf{1}_{\{u_1<
u_2<\cdots< u_{2p}\}}
(-1)^{N(nu_{_1})+N(nu_{_2})+\cdots+N(nu_{_{2p}})} du_1\cdots
du_{2p}\rbc,
\end{align*}
where in the latter equality we have used the symmetry of the
integrand. Taking into account that the two possible values
of random variable
$(-1)^{N(nu_{_1})+\cdots+N(nu_{_{2p}})}$
 only depend on the fact that the exponent is even or odd, we can write
 the latter expression above as
\beq
C_{p} n^{p}\lbc\ist\cdots\ist\mathbf{1}_{\{u_1<
u_2<\cdots< u_{2p}\}} E\Big((-1)^{\sum_{i=1}^{p}N(nu_{_
{2i}})-N(nu_{_{2i-1}})}\Big) du_1\cdots du_{2p}\rbc.
\label{eq:113}
\eeq
Using that
for $u_1< u_2<\cdots< u_{2p}$,  the random variables $N(nu_{_
{2i}})-N(nu_{_{2i-1}})$ are independent with Poisson distribution
of parameter $n(u_{2i}-u_{2i-1})$, we have that (\ref{eq:113}) is equal to
$$C_{p} n^{p}\lbc\ist\cdots\ist\mathbf{1}_{\{u_1<
u_2<\cdots<u_{2p}\}} e^{-2n\big[\sum_{i=1}^{p}
(u_{_{2i}}-u_{_{2i-1}})\big]} du_1\cdots du_{2p}\rbc.$$
This term can be bounded by
\begin{align*}
& C_{p} n^{p}\lbc\ist\cdots\ist\mathbf{1}_{\{u_1<
u_2\}}\cdots\mathbf{1}_{\{u_{2q-1}<u_{2p}\}}
e^{-2n\big[\sum_{i=1}^{p} (u_{_{2i}}-u_{_{2i-1}})\big]} du_1\cdots
du_{2p}\rbc.\\
& \qquad =C_{p} n^{p}\lbp\ist\int_s^{u_2}e^{-2n(u_2-u_1)}du_1\,du_2\rbp^{p}
= C_{p} n^{p}\lbp\ist \frac1{2n}(1-e^{-2n(u_2-s)})du_2\rbp^{p} \\
& \qquad \le C_{p} (t-s)^{p}.
\end{align*}
This concludes the proof.
\end{proof}

\noindent
\textbf{Step 2 (Kac-Stroock case):}

\begin{proposition}\label{mateixespai-Kac}
There exists a probability space $(\bar \oom,\bar{\mathcal F},\bar
P)$, a Brownian motion $\bbe$ defined on it and, for each $n\in \N$,
a process $\btn$ with the same law as $\theta^n$ in (\ref{ks-intro})
such that, for any $\nu\in (0,\frac14)$ and any $p\in\N$, 
\begin{equation}\label{boun-ks}
\sup_{t\in[0,1]}E\Big[\abs{\bbe(t)- \btn(t)}^{2p}\Big]\le
C_{p,\nu}\, n^{-p\nu},
\end{equation}
for some constant $C_{p,\nu}$.
\end{proposition}

\begin{proof}
First of all, it is clear that we can suppose $p(1/4-\nu)\ge 1$ (otherwise, we can use Jensen's inequality).
Then, following the lines of \cite[Section 2]{Griego-Heath-Ruiz}, we consider a probability space $(\bar \oom,\bar{\mathcal F},\bar P)$ with the following mutually independent objects defined on it:
\begin{itemize}
\item[(i)] a Brownian motion $\bbe$,
\item[(ii)] for each $n\in \N$, a sequence of independent random variables
$\{\xi_i^{(n)},\, i\in \N\}$ such that $\xi_i^{(n)}$ has an exponential distribution with parameter
$2\sqrt{n}$,
\item[(iii)] a sequence $\{k_i,\, i\in \N\}$ of independent random variables such
that $P\{k_i=1\}=P\{k_i=-1\}=1/2$.
\end{itemize}
The i.i.d.
random variables $k_1\xi_1^{(n)}, \,k_2\xi_2^{(n)},
\ldots,$ verify $E\big[k_i \xi_i^{(n)}\big]=0$ and $E\big[ \big(k_i\xi_i\pn\big)^2\big]=\frac1{2n}$. Therefore, with the same result of Skorokhod as the one quoted in the proof of Lemma \ref{mateixespai-Donsker}
(see \cite[p. 163]{Skorokhod}), there exists a sequence of
independent positive random variables
$\tau_1\pn,\,\tau_2\pn \dots$, such that
$\bbe(\tau_1\pn),\,\bbe(\tau_1\pn+\tau_2\pn) \dots$ have the
same law as $k_1\xi_1^{(n)}, \,k_1\xi_1^{(n)}+k_2\xi_2^{(n)} \dots$, respectively.
Moreover, it holds that
$$E\big[\tau_i\pn\big]=E\big[ \big(k_i\xi_i\pn\big)^2\big]=\frac1{2n},$$
and, for each $m\in\N$,
$$E\big[|\tau_i\pn|^m\big]\le C_m
E\big[|k_i\xi_i\pn|^{2m}\big]\le \frac{C_m}{n^m}.$$
Set $T^{(n)}_i:=\sum_{j=1}^i \tau^{(n)}_j$ and define
$$\widetilde{\tau}_i\pn:=n^{-1/2}\big|\bbe\big(T^{(n)}_i\big)-\bbe\big(T^{(n)}_{i-1}\big)\big| \quad , \quad \widetilde{T}^{(n)}_i:=\sum_{j=1}^i \widetilde{\tau}^{(n)}_j.$$
Then, let $\btn=\{\btn(t), \,t\ge0\}$ be a piecewise linear process given
on the grid $\widetilde{T}^{(n)}_1,\widetilde{T}^{(n)}_2,\ldots$ by
$$\btn(\widetilde{T}^{(n)}_i) := \bbe(T^{(n)}_i)\sim \sum_{j=1}^i
k_j\xi_j\pn,$$
and $\btn(0)=0$. The $\widetilde{\tau}^{(n)}_i$'s are independent
random variables exponentially distributed with parameter $2n$, and
it is proved in \cite{Griego-Heath-Ruiz} that the process $\btn$ thus defined has the same law as $\theta^n$.

\smallskip

\noindent
Now, to show (\ref{boun-ks}), we decompose the term $E\Big[ \abs{\bbe(t)- \btn(t)}^{2p}\Big]$ as the sum of the following
 two terms:
 $$E_1^n:=E\Big[\abs{\bbe(t)- \btn(t)}^{2p}\,\mathbf{1}_{\{
 t\in[0,\widetilde{T}^{(n)}_{8n}]\}}\Big]$$
 and
$$E_2^n:=E\Big[\abs{\bbe(t)- \btn(t)}^{2p}\,\mathbf{1}_{\{
 t>\widetilde{T}^{(n)}_{8n}\}}\Big].$$
 Let us first study $E_1^n$. If $t$ belongs to
 $A_{\ell}^n:=\big[\widetilde{T}^{(n)}_{\ell-1},\widetilde{T}^{(n)}_{\ell}\big)$ for some
 $\ell=1,\ldots, 8n$, we have that $$\btn(t)-\bbe(t)=\bbe\big(T^{(n)}_{\ell-1}\big)-\bbe(t)+\frac{t-\widetilde{T}^{(n)}_{\ell-1}}{\widetilde{\tau}^{(n)}_{\ell}}\Big\{\bbe\big(T^{(n)}_{\ell}\big)-\bbe\big(T^{(n)}_{\ell-1}\big)\Big\}.$$
So
\begin{eqnarray}
E_1^n &=& \sum_{\ell=1}^{8n} E\big[ |\bbe(t)-\btn(t)|^{2p} \mathbf{1}_{\{t\in A_{\ell}^n \}} \big]\nonumber\\
&\leq & C_p \sum_{\ell=1}^{8n} E\big[ |\bbe(T^{(n)}_{\ell-1})-\bbe(t)|^{2p} \mathbf{1}_{\{t\in A_{\ell}^n \}} \big]+C_p\sum_{\ell=1}^{8n}  E\big[ |\bbe(T^{(n)}_{\ell})-\bbe(T^{(n)}_{\ell-1})|^{2p} \mathbf{1}_{\{t\in A_{\ell}^n \}} \big]\nonumber\\
&\leq & C_p \sum_{\ell=1}^{8n} E\big[ |\bbe(T^{(n)}_{\ell-1})-\bbe(t)|^{2p} \mathbf{1}_{\{t\in A_{\ell}^n \}} \big]+C_p\, n^{1-p},\label{fix-1}
\end{eqnarray}
where, for the last inequality, we have used the fact that $\bbe(T^{(n)}_{\ell})-\bbe(T^{(n)}_{\ell-1}) \sim k_1 \xi_1$. Now,
 for any fixed $\ell\in \{1,\ldots,8n\}$, write
\begin{multline}\label{fix-2}
E\big[ |\bbe(T^{(n)}_{\ell-1})-\bbe(t)|^{2p} \mathbf{1}_{\{t\in A_{\ell}^n \}} \big]
=E\Big[ |\bbe(T^{(n)}_{\ell-1})-\bbe(t)|^{2p} \mathbf{1}_{\{t\in A_{\ell}^n \}}\mathbf{1}_{\{|t-T^{(n)}_{\ell-1}|\leq n^{-1/4}\}} \Big]\\
+E\Big[ |\bbe(T^{(n)}_{\ell-1})-\bbe(t)|^{2p} \mathbf{1}_{\{t\in A_{\ell}^n \}}\mathbf{1}_{\{|t-T^{(n)}_{\ell-1}|>n^{-1/4}\}} \Big].
\end{multline}
The first term in (\ref{fix-2}) can be bounded with the same argument as in the proof of Proposition \ref{mateixespai-Donsker} (see (\ref{fix-4})), which gives
$$E\Big[ |\bbe(T^{(n)}_{\ell-1})-\bbe(t)|^{2p} \mathbf{1}_{\{t\in A_{\ell}^n \}}\mathbf{1}_{\{|t-T^{(n)}_{\ell-1}|\leq n^{-1/4}\}} \Big]\leq C_p \, n^{-p/4}.$$
As far as the second term in (\ref{fix-2}) is concerned, we have
\begin{multline}\label{fix-3}
E\Big[ |\bbe(T^{(n)}_{\ell-1})-\bbe(t)|^{2p} \mathbf{1}_{\{t\in A_{\ell}^n \}}\mathbf{1}_{\{|t-T^{(n)}_{\ell-1}|>n^{-1/4}\}} \Big]\\
 \leq \Big\{ E\big[|\bbe(t)-\bbe(T^{(n)}_{\ell-1})|^{4p} \big] \Big\}^{1/2}\Big\{\bar{P}\big(t\in A_{\ell}^n , |t-T^{(n)}_{\ell-1}|>n^{-1/4} \big) \Big\}^{1/2},
\end{multline}
and since
$$
 E\Big[\big|\bbe\big(T^{(n)}_{\ell-1}\big)\big|^{4p}\Big]=E\Big[\big|\sum_{j=1}^{\ell-1}k_j\xi_j\pn\big|^{4p}\Big] \le C_{p}n^{2p}E\Big[\big|k_1\xi_1\pn\big|^{4p}\Big]\le C_{p},
$$
we only have to focus on the probability appearing in (\ref{fix-3}). To do so, let us notice that
\bean
\lefteqn{\bar{P}\big(t\in A_{\ell}^n , |t-T^{(n)}_{\ell-1}|>n^{-1/4} \big)}\\
 &\leq & \bar{P}\big( |t-\widetilde{T}^{(n)}_{\ell-1}|\leq \widetilde{\tau}^n_\ell, |t-T^{(n)}_{\ell-1}|>n^{-1/4} \big)\\
&\leq & \bar{P}\big( |t-\widetilde{T}^{(n)}_{\ell-1}|\leq \widetilde{\tau}^n_\ell,\widetilde{\tau}^{(n)}_\ell \leq \frac{1}{2} n^{-1/4}, |t-T^{(n)}_{\ell-1}|>n^{-1/4} \big)+\bar{P}\big( \widetilde{\tau}^{(n)}_\ell >\frac{1}{2} n^{-1/4}\big)\\
&\leq & \bar{P}\big( |t-\widetilde{T}^{(n)}_{\ell-1}|\leq \frac{1}{2} n^{-1/4}, |t-T^{(n)}_{\ell-1}|>n^{-1/4} \big)+C_p\, n^{-3p/2}\\
&\leq & \bar{P}\big( |T^{(n)}_{\ell-1}-\widetilde{T}^{(n)}_{\ell-1}|> \frac{1}{2} n^{-1/4} \big)+C_p\, n^{-3p/2}.
\eean
Then
$$\bar{P}\big( |T^{(n)}_{\ell-1}-\widetilde{T}^{(n)}_{\ell-1}|> \frac{1}{2} n^{-1/4} \big) \leq C_p \, n^{p/2} \Big\{ E\Big[\big|T^{(n)}_{\ell-1}-\frac{l-1}{2n}\big|^{2p} \Big]+E\Big[\big|\widetilde{T}^{(n)}_{\ell-1}-\frac{l-1}{2n}\big|^{2p} \Big]\Big\}\leq C_p \, n^{-p/2},$$
where we have used the same argument as in (\ref{ref-donsk-2}) to
get the last bound. Going back to (\ref{fix-1}), we deduce that
$E_1^n \leq C_p \, n \, n^{-p/4} \leq C_p \, n^{-\nu p}$, since $p$
is assumed to satisfy $p(\frac{1}{4}-\nu)\geq 1$.

\smallskip

\noindent
Eventually, we must deal with $E_2^n$. In fact, we have that
$$E_2^n\le \Big\{ E\Big[\big| \btn(t)-\bbe(t)\big|^{4p}\Big]\Big\}^{1/2}\Big\{
P\big(t>\widetilde{T}^{(n)}_{8n}\big)\Big\}^{1/2}\le C_{p}\Big\{
P\Big(t>\sum_{j=1}^{8n} \widetilde{\tau}^{(n)}_j\Big)\Big\}^{1/2},$$
where we have used Lemma
\ref{tightKac}. If we denote by $N_n$ a Poisson process with
intensity $2n$, we can write
$$
P\Big(t>\sum_{j=1}^{8n} \widetilde{\tau}^{(n)}_j\Big)\le  P\Big(1>\sum_{j=1}^{8n}\widetilde{\tau}^{(n)}_j\Big)\le P\big(
N_n(1)\ge 8n\big),$$
because the $\widetilde{\tau}^{(n)}_j$'s are independent random
variables exponentially distributed with parameter $2n$. The latter
probability can be bounded by using Stirling's inequality, as follows:
\begin{align*}
P\big( N_n(1)\ge 8n\big) & =\sum_{k=8n}^{\infty}e^{-2n}\frac{(2n)^k}{k!}
\le C\, e^{-2n} \sum_{k=8n}^{\infty}\frac{(2n)^k}{\sqrt{2\pi
k}\Big(\frac{k}{e}\Big)^k} \\
& =C\, e^{-2n}\sum_{k=8n}^{\infty}\Big(\frac{2en}{k}\Big)^k
\frac1{\sqrt{2\pi k}} \le C\,
e^{-2n}\sum_{k=8n}^{\infty}\Big(\frac{e}{4}\Big)^k\le Ce^{-2n}.
\end{align*}
This lets us conclude the proof.
\end{proof}

%%%%%%%%%%%%%%%%%%%%%%%%%%%%%%%%%%%%%%%%%%%%%%%%%%%%%%%%%%%%%%%%%%%%%%%%%%%%%%

\section{Appendix A: fractional Sobolev spaces}

We gather here some classical properties of the fractional Sobolev spaces $(\cb_{\al,p})_{\al\in \R,p\in \N}$, which are extensively used throughout the paper. We recall the notations $\cb_\al$ for $\cb_{\al,2}$ and $\cb$ for $\cb_0$. Let us first label the following well-known regularizing properties of the semigroup (see \cite{pazy}).

\begin{proposition}
Fix two parameters $\la <\al \in \R$. Then, for every $\vp \in \cb_{\la}$ and $t>0$,
\begin{equation}\label{regu-semi-1}
\|S_t \vp \|_{\cb_\al} \leq c\, t^{-(\al-\la)} \|\vp\|_{\cb_\la} \quad , \quad \|\Delta S_t \vp \|_{\cb_\al} \leq c\, t^{-1-(\al-\la)} \|\vp\|_{\cb_\la}.
\end{equation}
and for every $\psi \in \cb_\al$,
\begin{equation}\label{hold-sob}
\|S_t \psi-\psi \|_{\cb_\la} \leq c \, t^{\al-\la} \| \psi\|_{\cb_\al} \quad , \quad \|\Delta S_t \psi \|_{\cb_\la} \leq c \, t^{-1+(\al-\la)} \| \psi\|_{\cb_\al}.
\end{equation}
\end{proposition}

The next results are taken from the exhaustive book \cite{run-sick} on fractional Sobolev spaces. With the notations of the latter reference, our space $\cb_{\al,p}$ ($\al\in\R,p\in \N$) corresponds to $F^{2\al}_{p,2}$. Let us first report some properties regarding pointwise multiplication of functions. Due to the multiplicative perturbation in (\ref{equa-mild}), it is indeed natural that these results should intervene at some point. In the statement, the notation $ E \, \cdot\,  F \  \subset \ G $ must be understood as $\norm{\vp \cdot \psi}_G \leq c \, \norm{\vp}_E \norm{\psi}_F$
for every $\vp \in E,\psi \in F$.

\begin{proposition}
The following properties hold true:
\begin{enumerate}
\item (\cite[Section 2.4.4]{run-sick}) One has
\begin{equation}\label{sobolev-inclusion}
L^r(0,1) \subset \cb_{-\al} \quad  \text{if} \quad  \al \geq \frac{1}{2r}-\frac{1}{4},
\end{equation}
and in particular:
\[
L^p(0,1) \, \cdot \, \cb \  \subset \ \cb_{-\al} \quad \text{if} \quad \al \geq \frac{1}{2p}.
\]
\item (\cite[Section 4.6.1]{run-sick}) Let $\al_1<\al_2$ be such that $\al_1+\al_2>0$ and $\al_2>\frac{1}{4}$. Then
\begin{equation}\label{prod-sobol-1}
\cb_{\al_1} \, \cdot \, \cb_{\al_2} \  \subset \ \cb_{\al_1}.
\end{equation}
In particular, $\cb_\al$ is an algebra as soon as $\al> \frac{1}{4}$.
\item (\cite[Section 4.8.2]{run-sick}) Let $\al\geq 0$ and $p_1,p_2,p\geq 2$ be such that $2\al < \frac{1}{p_i}$ ($i\in \{1,2\}$) and $\frac{1}{p_1}+\frac{1}{p_2}=\frac{1}{p}$. Then
\begin{equation}\label{prod-sobol-2}
\cb_{\al,p_1} \, \cdot \, \cb_{\al,p_2} \  \subset \ \cb_{\al,p}.
\end{equation}
\end{enumerate}
\end{proposition}

\medskip

Let us also label here the classical Sobolev embedding
\begin{equation}\label{sobolev-inclusion-3}
\cb_{\al,p} \subset \cb_\infty \quad \text{if} \quad 2\alpha>\frac 1p,
\end{equation}
which yields in particular:
\begin{equation}\label{sobolev-inclusion-2}
\cb_{\al} \subset \cb_\infty \quad \text{as soon as} \quad \al >\frac14.
\end{equation}

\

Finally, in order to handle the non-linearity in (\ref{equa-mild}), we resort at some point to the following stability result for composition of functions (see \cite[Section 5.3.6]{run-sick}): if $f:\R\to \R$ is differentiable with bounded derivative, then for every $\al\in [0,\frac{1}{2}]$ and $\vp\in \cb_\al$,
\begin{equation}\label{nemy}
\norm{f(\vp)}_{\cb_\al} \leq c_f \lcl 1+\norm{\vp}_{\cb_\al} \rcl.
\end{equation}
Here, $f$ is also understood as its associated Nemytskii operator, i.e., $f(\vp)(\xi):=f(\vp(\xi))$.

%%%%%%%%%%%%%%%%%%%%%%%%%%%%%%%%%%%%%%%%%%%%%%%

\section{Appendix B: A priori estimates on the solution}

It only remains to prove the two a priori controls (\ref{apriori2-1}) and (\ref{apriori2-2}) for the solution $Y$ of (\ref{equa-mild}) 
(or equivalently (\ref{eq-base})). To do so, we will rely on the following result, taken from \cite[Lemma 6.5]{RHE}, 
and which extends the classical Garsia-Rodemich-Rumsey in two directions: 1) it covers the case of $\delha$-variations and 2) it 
applies to more general processes defined on the $2$-dimensional simplex $\cs_2=\{(t,s)\in [0,T]^2: t\geq s\}$.

\begin{lemma}\label{lem-grr}
Let $\der^\ast=\der$ or $\delha$. For every $\al,\be \geq 0$ and $p,q \geq 1$, there exists a constant $c$ such that, for any $R:\cs_2 \to\cb_{\al,p}$,
$$\cn[R;\cac_2^\be([0,T];\cb_{\al,p})] \leq c \lcl U_{\be+\frac{2}{q},q,\al,p}(R)+\cn[\der^\ast R;\cac_3^\be([0,T];\cb_{\al,p}] \rcl,$$
where
$$U_{\be,q,\al,p}(R)=\lc \int_{0\leq u<v\leq T} \lp \frac{\norm{R_{vu}}_{\cb_{\al,p}}}{\lln v-u\rrn^\be} \rp^q du dv \rc^{1/q}.$$
\end{lemma}

\begin{proof}[Proof of Lemma \ref{lem:apriori}]
In both cases, we will resort to the previous Lemma, which essentially reduces the problem to moment estimates. Thus, the following Burkholder-Davis-Gundy type inequality 
(borrowed from \cite[Lemma 7.7]{DaPrato-Zabczyk}) naturally comes into play: for every $\al \geq 0$, one has, by setting $U_0:=Q^{1/2}(\cb)$,
\begin{equation}\label{bdg}
E\Big[\Big\| \int_s^t S_{t-u}(f(Y_u) \cdot dW_u)\Big\|_{\cb_\al}^{2q} \Big]\leq c_q \Big( \int_s^t E\Big[ \|S_{t-u}(f(Y_u) \cdot \ast)\|_{HS(U_0,\cb_\al)}^{2q} \Big]^{\frac{1}{q}} du \Big)^q,
\end{equation}
where the notation $HS(U_0,\cb_\al)$ refers to the space of Hilbert-Schmidt operators defined on $U_0$ and taking values in $\cb_\al$. Note also that the family $(\la_k e_k)$ defines an orthonormal basis of $U_0$ and accordingly
\begin{equation}\label{hilb-schmi}
\|S_{t-u}(f(Y_u) \cdot \ast)\|_{HS(U_0,\cb_\al)}=\Big( \sum_k \la_k \|S_{t-u}(f(Y_u) \cdot e_k)\|_{\cb_\al}^2 \Big)^{1/2}.
\end{equation}
Now, to show that $Y\in \cacha^{2\eta}(\cb_\infty)$, observe first that for every $q\geq 1$ and any small $\ep >0$,
\begin{multline}\label{apri-inf}
E\big[ \norm{(\delha Y)_{ts}}_{\cb_\infty}^{2q}\big] \leq c_q \, E\Big[ \norm{(\delha Y)_{ts}}_{\cb_{\frac{1}{4}+\ep}}^{2q} \Big]\\
\leq c_q \Big\{ E\Big[\big\|\int_s^t S_{t-u}(f(Y_u) \cdot dW_u )\big\|_{\cb_{\frac{1}{4}+\ep}}^{2q} \Big] +E\Big[\Big(\int_s^t \norm{S_{t-u}(P\cdot f(Y_u)\cdot f'(Y_u)) }_{\cb_{\frac{1}{4}+\ep}} du \Big)^{2q} \Big]\Big\}.
\end{multline}
The second summand in (\ref{apri-inf}) is trivially bounded by $c_p \lln t-s\rrn^{2q(\frac{3}{4}-\ep)}$ since
$$\|S_{t-u}(P\cdot f(Y_u) \cdot f'(Y_u))\|_{\cb_{\frac{1}{4}+\ep}}\leq c \lln t-u\rrn^{-\frac{1}{4}-\ep} \|P\cdot f(Y_u) \cdot f'(Y_u)\|_\cb\leq c \lln t-u\rrn^{-\frac{1}{4}-\ep}.$$
As far as the first summand in (\ref{apri-inf}) is concerned, observe that
$$\sum_k \la_k \|S_{t-u}(f(Y_u) \cdot e_k)\|_{\cb_{\frac{1}{4}+\ep}}^2\leq c\sum_k \la_k \lln t-u\rrn^{-\frac12-2\ep}\|f(Y_u)\cdot e_k\|_\cb \leq c \lln t-u\rrn^{-\frac12-2\ep},$$
which, owing to (\ref{bdg}) and (\ref{hilb-schmi}), entails that
$$E\Big[\big\|\int_s^t S_{t-u}(f(Y_u) \cdot dW_u )\big\|_{\cb_{\frac{1}{4}+\ep}}^{2q} \Big]\leq c_q \lln t-s \rrn^{(\frac12-2\ep)q}.$$
Going back to (\ref{apri-inf}), we are in a position to apply Lemma \ref{lem-grr} and assert that $Y\in \cacha^{\frac14-}(\cb_\infty) \subset \cacha^{2\eta}(\cb_\infty)$ (we recall that $\eta$ is assumed to belong to $(0,\frac{1}{8})$). Note that since $\psi \in \cb_\ga$, these computations also prove that $\sup_{t\in [0,T]} E\big[\|Y_t\|^{2q}_{\cb_{\frac14+\ep}}\big] < \infty$ for $\ep >0$ small enough, which will be used in the sequel.

\smallskip

In order to show that $Y\in \cac^0(\cb_\ga)$, let us write, like in (\ref{apri-inf}),
\begin{eqnarray}
\lefteqn{E\Big[ \norm{(\delha Y)_{ts}}_{\cb_{\ga}}^{2q} \Big]}\nonumber\\
&\leq &c_q \Big\{ E\Big[\big\|\int_s^t S_{t-u}(f(Y_u) \cdot dW_u )\big\|_{\cb_{\ga}}^{2q} \Big] +E\Big[\Big(\int_s^t \norm{S_{t-u}(P\cdot f(Y_u)\cdot f'(Y_u)) }_{\cb_{\ga}} du \Big)^{2q} \Big]\Big\}\nonumber\\
&\leq & c_q \Big\{ E\Big[\big\|\int_s^t S_{t-u}(f(Y_u) \cdot dW_u )\big\|_{\cb_{\ga}}^{2q} \Big] +\Big( \int_s^t \lln t-u\rrn^{-\ga} du \Big)^{2q} \Big\}.\label{apri-ga}
\end{eqnarray}
Then one has successively
\bean
\lefteqn{E\Big[ \|S_{t-u}(f(Y_u) \cdot \ast)\|_{HS(U_0,\cb_\ga)}^{2q} \Big] \ =\ E\Big[ \Big( \sum_k \la_k \|S_{t-u}(f(Y_u) \cdot e_k)\|_{\cb_\ga}^2 \Big)^q \Big]}\\
&\leq & c_q \lln t-u\rrn^{-2q(\ga-\eta)} E\Big[ \Big(\sum_k \la_k \|f(Y_u) \cdot e_k \|_{\cb_\eta}^2\Big)^q \Big]\\
&\leq & c_q \lln t-u\rrn^{-2q(\ga-\eta)} E\Big[ \Big(\sum_k \la_k \|f(Y_u)\|_{\cb_{\frac14 +\ep}}^2 \| e_k \|_{\cb_\eta}^2\Big)^q \Big] \qquad \text{(use (\ref{prod-sobol-1}))}\\
&\leq & c_q \lln t-u\rrn^{-2q(\ga-\eta)} \big( \sum_k (\la_k \cdot k^{4\eta})\big)^q \big\{1+\sup_{t\in [0,T]} E\big[\|Y_t\|^{2q}_{\cb_{\frac14+\ep}}\big] \big\} \qquad \text{(use (\ref{nemy}))},
\eean
with $2(\ga-\eta)<1$ and $\sum_k (\la_k \cdot k^{4\eta}) <\infty$. Thanks to (\ref{bdg}) and (\ref{hilb-schmi}), we can go back to (\ref{apri-ga}) and deduce that $E\big[ \norm{(\delha Y)_{ts}}_{\cb_\ga}^{2q}\big] \leq c_q \lln t-s\rrn^{2q \ep}$ for some small $\ep >0$. Since $\psi \in \cb_\ga$, this proves in particular that $Y\in \cac^0(\cb_\ga)$.

\smallskip

Let us now turn to $K^Y$ and notice first that $\delha K^Y=L^W \der f(Y)$, so
\bean
\norm{(\delha K^Y)_{tus} }_\cb &\leq & \norm{L^W_{tu}}_{\cl(\cb,\cb)} \norm{\der(f(Y))_{us}}_\cb\\
&\leq & c_{W,f} \lln t-u\rrn^{\frac{1}{2}-\eta} \norm{(\der Y)_{us}}_\cb \qquad \text{(use (\ref{control-l-w}))}\\
&\leq & c_{W,f} \lln t-u\rrn^{\frac{1}{2}-\eta} \lcl \norm{(\delha Y)_{us}}_{\cb}+\norm{a_{us}Y_s}_\cb \rcl\\
&\leq & c_{W,f} \lcl \lln t-s \rrn^{\frac{1}{2}+\eta} \cn[Y;\cacha^{2\eta}(\cb_\infty)]+\lln t-s \rrn^{\frac{1}{2}-\eta+\ga} \cn[Y;\cac^0(\cb_\ga)] \rcl.
\eean
Besides, since
$$K^Y_{ts}=\int_s^t S_{t-u}(\der(f(Y))_{us}\cdot dW_u)+\int_s^t S_{t-u}(P\cdot f(Y_u)\cdot f'(Y_u)) \, du,$$
it is easy to check that $E\lc \norm{K^Y_{ts}}_\cb^{2q}\rc \leq c_q \lln t-s \rrn^{(1+4\eta)q}$ (use (\ref{bdg}) and (\ref{hilb-schmi}) as above). We are thus in a position to apply Lemma \ref{lem-grr} and conclude that $K^Y\in \cac_2^{\frac12+\eta}(\cb)$.

\end{proof}

%%%%%%%%%%%%%%%%%%%%%%%%%%%%%%%%%%%%%%%%%%%%%%%%%%%%%%%%%%%%%%%%%%%%%%%%%%%%

\section*{Acknowledgments}

We thank Samy Tindel for helpful discussions on the topic of the paper. We are also very grateful to two anonymous referees for their careful reading and suggestions, which have led to significant improvements in the presentation of our results.

\smallskip

Maria Jolis and Llu\'is Quer-Sardanyons 
are supported by the grant MCI-FEDER Ref. MTM2009-08869.

%%%%%%%%%%%%%%%%%%%%%%%%%%%%%%%%%%%%%%%%%%%%%%%%%%%%%%%%%%%%%%%%%%

\bibliography{mabiblio-trace,lluis}{}
\bibliographystyle{plain}

\end{document}